\documentclass[reqno,a4paper,11 pt]{amsart}
\usepackage{mathtools}
\usepackage{amssymb}
\usepackage{braket}
\usepackage[numbers]{natbib}
\usepackage[draft]{todonotes}
\usepackage{enumerate}
\usepackage{verbatim}
\usepackage[a4paper,top=2cm,bottom=2cm,left=2cm,right=2cm,bindingoffset=5mm]{geometry}

\newtheorem{theorem}{Theorem}
\newtheorem{proposition}{Proposition}
\newtheorem {lemma}{Lemma}

\newtheorem{remark}{Remark}

\newenvironment{system}
{\left\lbrace\begin{array}{@{}l@{}}}
{\end{array}\right.}
\DeclarePairedDelimiter{\abs}{\lvert}{\rvert}
\DeclarePairedDelimiter{\norm}{\lVert}{\rVert}
\DeclarePairedDelimiter{\br}{\lbrace}{\rbrace}
\newcommand{\opn}[1]{\operatorname{#1}}

\newcommand{\lsz}{\left\lbrace}
\newcommand{\psz}{\right\rbrace}
\newcommand{\lav}{\left|}                       % left absolute value
\newcommand{\pav}{\right|}                      % right absolute value
\newcommand{\lbr}{\left[}                       % left bracket
\newcommand{\rbr}{\right]}                      % right bracket
\newcommand{\lp}{\left(}                        % left parenthesis
\newcommand{\rp}{\right)}                       % right parenthesis

\newcommand{\R}{\mathbb{R}}
\newcommand{\E}{\mathbb{E}}
\newcommand{\mP}{\mathbb{P}}
\newcommand{\Fcal}{\mathcal{F}}      % caligraphic F for filtrations

\title[Necessary SMP for dissipative systems in infinite horizon]{Necessary stochastic maximum principle for dissipative systems on infinite time horizon}


\thanks{The first author has been supported by the Gruppo
Nazionale per l'Analisi Matematica, la Probabilit\`a e le loro
Applicazioni (GNAMPA) of the Istituto Nazionale di Alta Matematica
(INdAM) \\
The second author has been supported by the grant MIUR PRIN 2011 "Evolution differential problems: deterministic and stochastic
approaches and their interactions". }

\begin{document}

\author{Carlo Orrieri}
\address[C. Orrieri]{Dipartimento di Matematica, Universit\`a di Pavia. via Ferrata 1, 27100 Pavia, Italia}
\email{carlo.orrieri01@ateneopv.it}

\author{Petr Veverka}
\address[P. Veverka]{Dipartimento di Matematica, Politecnico di Milano. via Bonardi 9, 20133 Milano, Italia}
\email{petr.veverka@polimi.it}

\begin{abstract}
We develop a necessary stochastic maximum principle for a finite-dimensional stochastic control problem in infinite horizon under a polynomial growth and joint monotonicity assumption on the coefficients. The second assumption generalizes the usual one in the sense that it is formulated as a joint condition for the drift and the diffusion term. The main difficulties concern the construction of the first and second order adjoint processes by solving backward equations on an unbounded time interval. The first adjoint process is characterized as a solution to a backward SDE, which is well-posed thanks to a duality argument. The second one can be defined via another duality relation written in terms of the Hamiltonian of the system and linearized state equation. Some known models verifying the joint monotonicity assumption are discussed as well.
\end{abstract}

\maketitle

\section{Introduction}
The study of the stochastic maximum principle (SMP in short) is a current field of research motivated by the interest in finding necessary (and sufficient) conditions for optimality for stochastic control problems. The general idea of the SMP consists in associating to every controlled trajectory an adjoint equation which is backward in time. Its solution, called a dual process (which is, in fact, a pair of processes), is shown to exist under some appropriate conditions and plays a role of ``generalised Lagrange multipliers''. The SMP is a variational inequality formulated by means of the state trajectory and the dual process. It is satisfied by any optimal control and, usually, adding some convexity assumptions, it fully characterizes the optimality. Moreover, if the control enters the diffusion, the irregularity in time of the noise forces to introduce a second adjoint process which is strictly related to the Lyapunov equation for the first variation of the state.

The first general formulation of the SMP is due to Peng \cite{peng1990general} for finite dimensional systems. After this seminal work, there has been a large number of works on this subject, both in finite and infinite dimensions for different formulation of the control problem. SMP in infinite dimension
has been studied e.g. in Tang and Li \cite{tang1993maximum}, Fuhrman, Hu and Tessitore \cite{fuhrman2013stochastic}, Du and Meng \cite{du2013maximum}, L\"u and Zhang \cite{lu2012general} whereas some of the results in finite dimension comprise: Jump control: Tang and Li \cite{tang1994necessary}, \O ksendal and Sulem \cite{oksendal2005applied};
Singular control: Bahlali and Mezerdi \cite{bahlali2005general}, Dufour and Miller \cite{dufour2006maximum}, \O ksendal and Sulem \cite{oksendal2011singular}; Impulse control: Wu and Zhang \cite{wu2012maximum};
Delayed controlled systems: \O ksendal, Sulem and Zhang \cite{oksendal2011optimal};
Near-optimal control: Zhou \cite{zhou1998stochastic} and many others. \bigskip

This paper is a natural continuation of \cite{maslowski2014sufficient} on one side and \cite{orrieri2013stochastic} on the other. Our aim is to control the behaviour of a dissipative system in an unbounded time interval and to provide necessary conditions for optimality. If $W_t$ is a $d$-dimensional Brownian motion, the equation for the state can be written in the form
\begin{equation*}
X_t = x + \int_0^t b(X_s,u_s)ds + \int_0^t \sigma(X_s,u_s)dW_s,
\end{equation*}
$\mP-$a.s. for all $t \in [ 0,+\infty )$ and all $x \in \R^n$.
The objective is to minimize a discounted functional
\begin{equation*}
J\lp u(\cdot)\rp  = \E \int^{+\infty}_0 e^{-rt} f(X_t,u_t)dt,
\end{equation*}
 over all admissible controls $u(\cdot)$ with values in a general separable metric space $(U,d)$. Let us remark that this general assumption on the space of control actions allows us to consider a broad class of controls, such as bang-bang controls, which are excluded from the classic convexity framework. On the other hand, it necessarily forces us to formulate the SMP via a spike perturbation argument.

The analysis of the problem is based on the well-posedness of the state equation under a \emph{joint (or global) monotonicity} assumption.
For any $ x,y \in \mathbb{R}^n$ and fixed $p >0$, there exists $c_p \in \R$ such that
\begin{equation}
\braket{b(x,u) - b(y,u),x-y} + p \norm{\sigma(x,u) - \sigma(y,u)}^2_2 \leq c_p \abs{x-y}^2, \quad \forall u \in U,
\end{equation}
where $U$ is the space of control actions. For a detailed exposition of SDEs with this property see \citep{cerrai2001second}. This condition is a generalization of the usual dissipativity condition which involves only the drift of the equation and allows us to consider a larger class of concrete examples. Informally, there is a balance between the dissipativity of the drift and the noise term. If the drift term is dissipative enough, the diffusion term can grow in a polynomial way, instead of being globally Lipschitz. As a reference for this general assumption, see \cite{cerrai2001second} and \cite{prevot2007concise}. A list of several important examples satisfying joint monotonicity condition is given further in this paper.
However, let us notice that many interesting equations do not satisfy a global monotonicity assumption (see \cite{cox2013local} for a selection of examples) and the formulation of a version of the SMP for these systems could be a subject of a future work.\smallskip

In the first step of our analysis we have to deal with a partially-coupled system of the state equation and the first adjoint equation. The delicate question here consists in giving a precise meaning to the solution of the following backward SDE
\begin{equation}
dp_t = - \Big[D_x H(\bar{X}_t,\bar{u}_t,p_t,q_t) - rp_t \Big]dt + q_t dW_t,
\end{equation}
where $H(x,u,p,q) = \braket{p,b(x,u)} + \opn{Tr}\left[ q^T \sigma(x,u) \right] - f(x,u)$ is the Hamiltonian of the system and
$\lp \bar{X},\bar{u} \rp$ is an optimal pair.

In general, the behaviour at infinity of BSDEs is not easy to understand and different approaches and approximations are proposed in several papers. In our setting, we are able to tackle the problem showing that the adjoint equation preserves, in some sense, the dissipativity of the state. Using a duality argument and the same technique as in \cite{peng2000infinite} we can show that there exists a solution in some exponentially weighted space.
It turns out that for the analysis of the state and adjoint equations the condition on the discount factor (forming both the functional and the weight
$e^{-rt}$ in the definition of exponential weighted space) is given by some formula in terms of the joint monotonicity constant $c_p \in \R$. Nevertheless, due to the framework of SMP (use of the spike variation techniques) and, more importantly, due to the form of the polynomial
growth assumption one has to assume implicitly $r$ positive so that the polynomial bound is integrable with the weight. \smallskip

As already mentioned, the presence of the control in the diffusion term makes a second adjoint process to appear. In this case, there exists a formal matrix-valued BSDE which represents the process, but due to the lack of dissipativity of the equation, it seems to be impossible to obtain an a priori estimate of the solution. To solve the problem we follow the idea of Fuhrman et al. in \cite{fuhrman2013stochastic} and we define the first component of the second adjoint process $P_t$ as a bilinear form defined via the relation
\begin{equation}\label{eq:intro duality}
\braket{P_t \eta , \gamma} := \E^{\Fcal_t} \int_t^\infty e^{-r(s-t)} \braket{D_x^2 H(\bar{X}_s,\bar{u}_s,p_s,q_s) y^{t,\eta}_s ,y^{t,\gamma}_s} ds,
\end{equation}
where $( y^{t,\eta}_s )_{s \geq t}$ is the solution of the linearized state equation starting from $\eta$ at time $t$. Note that the second component $Q_t$ does not appear in the definition of the SMP. Proceeding this way it is not necessary to define and solve the second adjoint equation (i.e. finding the couple $(P,Q)$) but it is sufficient to identify only the process $P$ via the equality (\ref{eq:intro duality}). Notice that our definition of $P$ is similar to the notion of transposition solution presented in \cite{lu2012general}. Nevertheless, if the diffusion term $\sigma$ is Lipschitz, $P$ can be indeed identified as a unique solution to a matrix-valued BSDE which, in fact, inherits the monotonicity property from the state equation. The formulation of the SMP, in this case, follows by similar arguments as in \cite{orrieri2013stochastic} but with an extension to the infinite horizon setting. \smallskip

Having in mind the form of the Hamiltonian of the system, the final step (and the main result, Theorem \ref{main_thm}) is to derive a necessary condition for optimality. Let us suppose that $(\bar{X},\bar{u})$ is an optimal pair, then
for every $v \in U$, the following variational inequality has to hold $\mP \otimes dt-$a.e.
\begin{equation*}
H(\bar{X}_t,v,p_t,q_t) -  H(\bar{X}_t,\bar{u}_t,p_t,q_t) + \frac{1}{2}\sum^d_{j=1} \Big<P_t\left( \sigma^j(\bar{X}_t,v) - \sigma^j(\bar{X}_t,\bar{u}_t) \right), \sigma^j(\bar{X}_t,v) - \sigma^j(\bar{X}_t,\bar{u}_t)\Big> \le 0.
\end{equation*}
This variational inequality can be rewritten in terms of so called $\mathcal{H}$-function defined by
\begin{equation*}
\begin{split}
\mathcal{H}(x,u) &:= H(x,u,p_t,q_t) - \dfrac{1}{2}\opn{Tr}\bigl( \sigma(\bar{X}_t,\bar{u}_t)^TP_t\sigma(\bar{X}_t,\bar{u}_t) \bigr)\\
&\quad + \dfrac{1}{2}\opn{Tr}\bigl[ (\sigma(x,u) - \sigma(\bar{X}_t,\bar{u}_t))^TP_t(\sigma(x,u) - \sigma(\bar{X}_t,\bar{u}_t)) \bigr],
\end{split}
\end{equation*}
in the following manner

\begin{equation*}
\mathcal{H} \lp \bar{X}_t,\bar{u}_t \rp = \max_{v \in U} \mathcal{H}\lp \bar{X}_t, v \rp, \quad  \mP \otimes dt-\text{a.e.}
\end{equation*}

The paper is organized as follows. The second and third sections contain the basic assumptions, the formulation of the discounted problem and a list of motivating examples. In section 4, we study well posedness of the state equation. The fifth section concerns with the application of the spike variation technique to our problem. In the two next sections,  6 and 7, the two adjoint processes are studied. The first adjoint BSDE is solved by
approximation and a duality argument whereas, the construction of the process $P$ is described without any potential relation to some BSDE. A precise statement of the main theorem (Theorem \ref{main_thm}) is presented in section 8 and its proof is given. Last, in the Appendix we provide some technical proofs and we quote the actual restriction for the discount factor including the final one used in formulation of the main result.

\section{Assumptions and preliminaries}

Let $W = \lbrace{W^1_t,\ldots,W^d_t\rbrace}_{t\geq 0}$ be a standard $d$-dimensional Brownian motion defined on some complete filtered probability space $(\Omega, \mathcal{F}, \lp \mathcal{F}_t \rp _{t \geq 0}, \mP)$. The filtration $\lp \mathcal{F}_t \rp _{t \geq 0}$ is assumed to be the
canonical filtration of $W$ completed by $\mP-$null sets.
The space of control actions is a general metric space $U$ endowed with its Borel $\sigma$-algebra $\mathcal{B}(U)$. The class of admissible controls is defined as follows
\begin{equation*}
\mathcal{U}:= \br{u(\cdot): \R_+ \times \Omega \rightarrow U : u(\cdot) \text{ is } \lp \mathcal{F}_t \rp_{t\geq 0}-\text{progressive}}.
\end{equation*}

\noindent For $r \in \R,\ p > 1$ and a Banach space $E$, we define
\begin{align}
L^{p,-r}_{\mathcal{F}}(\R_+;E):= & \Big\{v(\cdot):\R_+ \times \Omega \rightarrow E: v(\cdot) \text{ is } \lp \mathcal{F}_t \rp_{t\geq 0}-\text{progressive} \Big. \nonumber \\
& \qquad \qquad \Big. \text{ and } \E\int_0^{\infty}e^{-rt}\norm{v_t}^p_E dt < \infty \Big\}.
\end{align}

\noindent We want to study an infinite horizon stochastic control problem in $\R^n$ of the form
\begin{equation}\label{SDE}
\begin{system}
dX_t = b(X_t,u_t)dt + \sigma(X_t,u_t)dW_t, \qquad \forall t\geq 0,\\
X_0 = x,
\end{system}
\end{equation}
where $x \in \R$ and $u(\cdot)$ is an admissible control. The discounted functional to be minimized is given by

\begin{equation}\label{functional}
J\lp u(\cdot)\rp  = \E \int^{+\infty}_0 e^{-rt} f(X_t,u_t)dt.
\end{equation}

\noindent By $|\cdot|$ we denote the Euclidean norm on $\R^n$, $\norm{\cdot}$ stands for a Frobenius norm on $\R^{n \times d}$ and, finally, $\norm{\cdot}_2$ denotes the Hilbert-Schmidt norm on $\R^{n \times n}$. By $\mathcal{S}^n$ we denote the set of symmetric matrixes $\R^{n \times n}$. $\chi_A$ denotes the characteristic function of a set $A$. \bigskip

{\bf{Hypotheses:}}

\begin{itemize}

\item[(H1)] $(U,d)$ is a separable metric space. \bigskip

\item[(H2)] ({\bf{Polynomial growth}}) The vector field $b: \R^n \times U \rightarrow \R^n$ is $\mathcal{B}(\R^n)\otimes \mathcal{B}(U)$-measurable and the map $x \mapsto b(x,u)$ is of class $\mathcal{C}^2$. Moreover, there exists $m\geq 0$ such that

\begin{equation} \label{eq:grad b}
\sup_{u \in U}\sup_{x \in \mathbb{R}^n} \frac{\abs{D^{\beta}_x b(x,u)} }{1+ \abs{x}^{2m+1}} < +\infty,\qquad \abs{\beta} =0,1,2.
\end{equation}

\item[(H3)] ({\bf{Polynomial growth}}) The mapping $\sigma: \R^n \times U \rightarrow \R^{n\times d} $ is measurable with respect to $\mathcal{B}(\R^n) \otimes \mathcal{B}(U)$. Moreover the map $x \mapsto \sigma(t,x,u)$ is $\mathcal{C}^2(\R^n;\R^{n\times d})$ and there exists $m$ (same as in (H2)) such that

\begin{equation} \label{eq:grad sigma}
\sup_{u\in U}\sup_{x \in \mathbb{R}^n} \frac{\norm{D^{\beta}_x \sigma(x,u)}_2 }{1+ \abs{x}^{m}} < +\infty,\qquad \abs{\beta} =0,1,2.
\end{equation}
\smallskip
\item[(H4)] ({\bf{Joint monotonicity}}) Let $p >0$. Then there exists $c_p \in \mathbb{R}$ such that,

\begin{equation} \label{eq:joint dissipativity}
\braket{D_xb(x,u)y,y} + p\norm{D_x\sigma(x,u)y}^2_2 \leq c_p \abs{y}^2, \qquad x,y \in \mathbb{R}^n, u \in U.
\end{equation}
\smallskip
\item[(H5)] ({\bf{Cost}}) The function $f: \R^n \times U \rightarrow \R$ is $\mathcal{B}(\R^n)\otimes \mathcal{B}(U)$-measurable and the map $x \mapsto f(x,u)$ is of class $\mathcal{C}^2$. Moreover, there exists $l \geq 0$ such that

\begin{equation} \label{eq:grad l}
\sup_{u \in U}\sup_{x \in \mathbb{R}^n} \frac{\abs{D^{\beta}_x f(x,u)}}{1+ \abs{x}^l} < +\infty,\qquad \abs{\beta} = 0,1,2.
\end{equation}

\end{itemize}

\begin{remark}
\normalfont
\begin{itemize}
\item[]
\item[(a)] In \cite{cerrai2001second}, the form of (H2) and (H3) is given in a stronger way. For our purposes the above formulation is sufficient.  \medskip

\item[(b)] It is possible to show that condition \eqref{eq:joint dissipativity} implies that for any $ x,y \in \mathbb{R}^n$
\begin{equation} \label{eq:joint dissipativity 2}
\braket{b(x,u) - b(y,u),x-y} + p \norm{\sigma(x,u) - \sigma(y,u)}^2_2 \leq c_p \abs{x-y}^2, \quad u \in U.
\end{equation}
and the two conditions are equivalent for coefficients twice differentiable (in $x$) which, in fact, is our case.\medskip

\item[(c)] The joint monotonicity condition \eqref{eq:joint dissipativity 2} also implies the so called coercivity condition (see i.e. \cite{prevot2007concise}, page 44). Indeed, let us fix $y= 0$, then \eqref{eq:joint dissipativity 2} reduces to
\begin{equation}
\braket{b(x,u) - b(0,u), x} + p\norm{\sigma(x,u)-\sigma(0,u)}_2^2 \leq c_p\abs{x}^2,
\end{equation}
and
\begin{equation*}
\begin{split}
\norm{\sigma(x,u)-\sigma(0,u)}_2^2 &\geq \big| \norm{\sigma(x,u)}_2 - \norm{\sigma(0,u)}_2\big|^2 \\
&= \norm{\sigma(x,u)}_2^2 + \norm{\sigma(0,u)}_2^2 - 2\norm{\sigma(x,u)}_2\norm{\sigma(0,u)}_2 \\
&\geq \norm{\sigma(x,u)}_2^2 + \norm{\sigma(0,u)}_2^2 - \varepsilon\norm{\sigma(x,u)}_2^2 - \dfrac{\norm{\sigma(0,u)}_2^2}{\varepsilon} \\
& = (1-\varepsilon)\norm{\sigma(x,u)}_2^2 + \norm{\sigma(0,u)}_2^2 - \dfrac{\norm{\sigma(0,u)}_2^2}{\varepsilon},\ \forall \varepsilon \in (0,1).
\end{split}
\end{equation*}
Then after easy computation we obtain that
\begin{equation}\label{eq:coercivity type estimate}
\braket{b(x,u),x} + p(1-\varepsilon)\norm{\sigma(x,u)}_2^2 \leq \tilde{K}_{p}(1+ \abs{x}^2),\quad \varepsilon \in (0,1).
\end{equation}
Let us note that $\abs{b(0,u)} + \abs{\sigma(0,u)} \leq C$ due to the polynomial growth (H2)-(H3), hence $\tilde{K}_{p}$ can be chosen as
$\tilde{K}_{p} = \max\{c_p + 1/2, C^2/2 ,1\}.$ \medskip

\item[(d)] The above Hypotheses (H2)-(H5) can be generalized to the situation of time dependent stochastic coefficients under natural assumptions without any influence on the main result.
\end{itemize}
\end{remark}

\section{Motivations and examples}
Apart from the classical Lipschitz setting, there are two usual sets of assumptions which assure global existence and uniqueness of the solution to an SDE. The first one consists in local Lipschitz property of the coefficients along with the so called coercivity condition
\begin{equation}\label{eq:global_coercivity}
\braket{b(x),x} + \frac{1}{2}\norm{\sigma(x)}_2^2 \leq K(1+ \abs{x}^2),
\end{equation}
for all $x \in \R^n$ and some $K \in \R$.
The second set comprises some dissipativity assumptions on the drift, still with Lipschitz diffusion term. The dissipativity is expressed by
\begin{equation*}
\braket{b(x)- b(y),x-y} \leq K\abs{x-y}^2,
\end{equation*}
for all $x,y \in \R^n$ and some $K \in \R$. \smallskip

Another step further in this direction is the so called global monotonicity assumption, which is formulated as a joint condition for drift and diffusion
\begin{equation}\label{eq_dissip_uncontrolled}
\braket{b(x) - b(y),x-y} + \frac{1}{2} \norm{\sigma(x) - \sigma(y)}^2_2 \leq c \abs{x-y}^2,
\end{equation}
for all $x,y \in \R^n$. It is important to mention that the joint monotonicity property immediately implies the dissipativity of the drift but not necessarily global Lipschitzianity of the diffusion part. Moreover, the joint monotonicity \eqref{eq_dissip_uncontrolled} also implies the coercivity property \eqref{eq:global_coercivity}.

It turns out that for the purposes of the SMP it is natural to strengthen the global monotonicity assumption in the following form:
for all fixed $p>0$ there exists $c_p \in \R$ such that
\begin{equation}\label{eq_dissip_uncontrolled1}
\braket{b(x) - b(y),x-y} + p\norm{\sigma(x) - \sigma(y)}^2_2 \leq c_p \abs{x-y}^2,
\end{equation}
for all $x,y \in \R^n$.
This is motivated by the attempt to solve not only the state equation, but also the first and second variation equations and to derive some appropriate estimates of higher moments of the solutions. Another natural assumption in this framework is the polynomial growth of the coefficients along with their derivatives. This is fundamental in order to choose the correct discount factor in the definition of the weighted spaces $L_{\mathcal{F}}^{2,-r}(\R_+; \R^n)$ that we are going to use. Let us also mention that these polynomial bounds allow us to prove the local Lipschitzianity of the coefficients of the state equation.

To conclude, notice that the freedom in choosing $p$ in the definition of \eqref{eq_dissip_uncontrolled1} implies the existence of a link between the growth of the diffusion term and the drift (compare (H2) and (H3), see also \cite{cerrai2001second}). For example, to gain a quadratic growth in the diffusion we have to require the system to be more dissipative. Concretely, one such an example is
\begin{equation}\label{eq_x^5}
dX_t = \left[ X_t - X_t^5\right]dt + X_t^2dW_t.
\end{equation}

A more general framework is presented in \cite{prevot2007concise} and \cite{lan2014new} where the authors do not ask for polynomial growth of the coefficients and present a \emph{weak local} version of the global monotonicity assumption along with a \emph{weak} coercivity assumption. By \emph{weak} we mean the presence of stochastic coefficients instead of constant ones in the definitions of the conditions, see \cite{prevot2007concise} for a detailed exposure. \medskip

\noindent After this preliminary discussion we also present some concrete models. \smallskip

\begin{enumerate}
\item {\bf Polynomial model:} As we have discussed above, a one dimensional model given by the SDE
\begin{equation*}
dX_t =  \left[ - X^{2m+1}_t + \sum_{i=1}^{2m} a_i X_t^i \right]dt  +  \left[ \sum_{i=1}^m b_i X^i_t\right] dW_t; \qquad  X_0 = x \in \R,
\end{equation*}
for some $a_i, b_i \in \R$, is the simplest example we have in mind. Let us notice that we can easily generalize the model in a way so that these polynomials are upper bounds for some more general (but locally Lipschitz) functions satisfying the joint monotonicity condition. \bigskip

\item {\bf Population growth models:} A model in $\R$ given by the SDE

\begin{equation*}
dX_t = \alpha X_t h\lp X_t \rp dt  + \sigma X_t dW_t; \qquad X_0 = x >0,
\end{equation*}
where $h(x) = 1 - \beta \ln(x)$, for so called Gompertz growth models (tumor growth models etc.) or
$h(x) =  1 - \beta x $, for so called logistic growth models (population dynamics models etc.). A detailed discussion of the controlled logistic model on infinite time horizon can be found \cite{maslowski2014sufficient}.
In both cases, $\alpha >0$ is the speed of growth and $\beta>0$ represents some saturation level of the system. It can be shown by Lyapunov techniques that the solution is positive and an explicitly analytic formula can be found by linearizing the two equations.
It is important to mention that our version of SMP covers the case of controlled logistic models (in full generality) whereas the controlled Gompertz model can be treated only with uncontrolled diffusion. This fact is due to the lack of polynomial growth condition needed in (H2) and the second variation equation might not be well posed. The same argument holds for another generalizations of the
two population models with different choices of diffusion term ($\sigma x (1-\ln(x)), \sigma \sqrt{x(1-\ln(x))}, \sigma \sqrt{x(1-x)}$ etc.). \bigskip

\item {\bf Gradient flow model with stochastic perturbation:} Let $K$ be a convex open subset of $\R^d$ and $\varphi: K \subseteq \R^d \rightarrow \R$ a $\lambda$-convex function, i.e.
\begin{equation*}
\varphi((1-\alpha)x_0 + \alpha x_1) \leq (1-\alpha)\varphi(x_0) + \alpha \varphi(x_1) - \frac{\lambda}{2}\alpha(\alpha -1)\abs{x_1 - x_0}^2,
\end{equation*}
for every  $x_0,x_1 \in \R^d$ and $\alpha \in [0,1]$.
Then we can study a SDE of the form
\begin{equation*}
dX_t = -\nabla\varphi(X_t)dt + \sigma(X_t)dW_t,
\end{equation*}
provided that $\varphi$ is of class $C^1$. In fact, $\lambda$-convexity (with $C^1$-regularity) is equivalent to
\begin{equation*}
\braket{\nabla\varphi(y),y-x} - \frac{\lambda}{2}\abs{y-x}^2 \geq \varphi(y) - \varphi(x) \geq \braket{\nabla\varphi(x),y-x} + \frac{\lambda}{2}\abs{y-x}^2
\end{equation*}
which in particular implies that $\nabla\varphi$ is $\lambda$-dissipative.
If we ask $\sigma$ to be Lipschitz, then \eqref{eq:joint dissipativity 2} is satisfied.

Some possible choices of $\varphi(\cdot)$ are:
\begin{itemize}
\item Take $\lambda = 0$ and $\varphi(x) = \abs{x}^{2k}$ convex with the derivative $2k\abs{x}^{2k-2}x$;
\item (Double-well potential) Let $d =1$ and consider $\varphi(x) = (x^2-1)^2$, which is not convex ($\pm 1$ are minima) but $\lambda$-convex.
\item Let $d=2$ and consider the following dynamics
\begin{equation*}
\begin{system}
dX_t = -X_tdt + X_tY_t^2(1+ X_t^2)^{-2}dt + \sigma dW_t^{(1)}, \\
dY_t = -Y_t(1 + X_t^2)^{-1}dt + \sigma dW_t^{(2)}.
\end{system}
\end{equation*}
Here the energy has the following form $\varphi = \varphi_{1} + \varphi_{2}$, where
\begin{equation*}
\varphi_{1}(x,y) = \frac{x^2}{2} \qquad \text{ and } \qquad \varphi_{2}(x,y)=\frac{y^2}{2(1+x^2)}.
\end{equation*}
The difference between this case and the previous one is that here, the energy has not isolated minima but rather forms a sub-manifold (i.e. the $x$-axis).
\end{itemize}
\end{enumerate}

\section{State equation}
In this section we provide the existence and uniqueness theorem for the state equation \eqref{SDE}.
The classical proof of such theorem under our Hypotheses (H1)-(H4) goes along the lines as in \cite{cerrai2001second}, Section 1.2. for a solution in the space $L^2_{\Fcal} \lp [0,T]; \R^n\rp$ (thus not in exponentially weighted space $L^{2,-r}_{\Fcal} \lp [0,T]; \R^n\rp$). Nevertheless, by these arguments one can not obtain a
contraction from $L^{2,-r}_{\Fcal} \lp [0,+\infty); \R^n\rp$ to itself (not even locally in time).
Hence an another approach has to be chosen. More specifically, we will scale the original equation in a way so that the classical result
from \cite{cerrai2001second} can be applied. We stress that, in this case, the bound for the discount factor depends on $c_p$ which
can be also negative.

\begin{theorem}\label{th:1.1}
Let Hypotheses (H1)-(H4) hold. Then for every admissible control $u(\cdot)$ there is a unique solution process $(X_t)_{t\geq0}$ to SDE (\ref{SDE}) with
$ \sup_{t \in [0,T]} \E \lbr e^{-rt}|X_t|^2 \rbr < +\infty, $ for each $T>0$ and for $r>2c_{1/2}$. Moreover, for all $q \geq 2$ and for $r>2c_{q-1}$ it holds
\begin{equation}\label{eq:state estimate 1}
\E \int^{\infty}_0 e^{-rtq} |X_t|^{2q} dt \leq C_1|x|^{2q},\quad \text{for some constant}\ C_1 = C_1(q)>0,
\end{equation}
where $c_{q-1}$ is the joint monotonicity constant in (\ref{eq:joint dissipativity}). In addition, for $q \in \left[ \frac{1}{2},2 \right)$ and $r>4c_1$, it holds
\begin{equation}\label{eq:state estimate 3}
\E \int^{\infty}_0 e^{-rt} |X_t|^{2q} dt \leq C_2 |x|^4, \quad \text{for some constant}\ C_2 = C_2(q)>0.
\end{equation}
\end{theorem}

\begin{proof}
Without loss of generality we prove the result for uncontrolled case (which can be easily converted to the controlled one).
The idea is to apply the result from \cite{cerrai2001second} to a transformed equation which corresponds to our original equation.
Assume for a moment that we already have a process $X$ satisfying the dynamics given by \eqref{SDE} and define $\tilde{X}_t \equiv e^{-\frac{r}{2}t}X_t$. Then $\tilde{X}$ solves
\begin{equation}\label{eq:SDE tilde}
\begin{system}
d\tilde{X}_t = -\frac{r}{2}\tilde{X}_t +  e^{-\frac{r}{2}t} b\Bigl( e^{\frac{r}{2}t} \tilde{X}_t \Bigr)dt
+ e^{-\frac{r}{2}t}\sigma\Bigl( e^{\frac{r}{2}t} \tilde{X}_t \Bigr)dW_t, \qquad \forall t\geq 0,\\
\tilde{X}_0 = x.
\end{system}
\end{equation}

\noindent Denoting $\tilde{b}(t,x) = -\frac{r}{2}x +  e^{-\frac{r}{2}t} b\bigl( e^{\frac{r}{2}t} x \bigr)$ and
$\tilde{\sigma}(t,x) = e^{-\frac{r}{2}t} \sigma \bigl( e^{\frac{r}{2}t} x \bigr)$ it is easy to check that
$\tilde{b}, \tilde{\sigma}$ also satisfy assumptions (H1)-(H4). Differentiability and polynomial growth (H2)-(H3) are evident whereas joint
monotonicity (H4) holds in the following sense
\begin{align}\label{eq:forward transform1}
\braket{\tilde{b}(t,x) - \tilde{b}(t,y), x-y} + p ||\tilde{\sigma}(t,x) - \tilde{\sigma}(t,y) ||^2_2
\leq \tilde{c}_p |x-y|^2,\qquad \tilde{c}_p = c_p - \frac{r}{2}.
\end{align}
\noindent Hence, due to \cite{cerrai2001second}, Section 1.2 there exists a unique predictable process $\tilde{X}$ solving (\ref{eq:SDE tilde}) which satisfies $\sup_{t \in [0,T]} \E |\tilde{X}_t|^2 < +\infty,$ for each $T>0$. But this means that there actually exists a process $X$ solving \eqref{SDE} with the following integrability
property
\begin{equation}\label{eq:forward transform2}
\sup_{t \in [0,T]} \E \bigl[ e^{-rt}|X_t|^2 \bigr]= \sup_{t \in [0,T]} \E |\tilde{X}_t|^2 < +\infty,\qquad \forall T>0.
\end{equation}

\noindent The last step is proving the desired exponentially weighted integrability.
Denote the exit time $\tau_K := \inf \{t\geq 0: |\tilde{X}_t| \geq K \}$ for each $K >0$ with the usual convention $\inf\{\emptyset\} := +\infty$.
It is easy to show that $\tau_K \nearrow \tau =+\infty$, for $K\rightarrow +\infty$ a.s.\\
\newline
Proof of estimate (\ref{eq:state estimate 1}): Let us fix $q\geq2$ and apply the It\^o formula to $\abs{\tilde{X}_{t \wedge \tau_K}}^{2q}$. We also denote $\tilde{a}_s = \tilde{\sigma}(s, \tilde{X}_s)\tilde{\sigma}( s, \tilde{X}_s)^T$. Then we obtain
\begin{align}\label{eq:state estimate Ito}
\E \abs{\tilde{X}_{t \wedge \tau_K}}^{2q} &= |x|^{2q}
+ 2q\E \int^t_0 \chi_{\{\tau_K \geq s \} } \abs{\tilde{X}_s}^{2(q-1)}\left(\braket{\tilde{X}_s, \tilde{b}\left(s, \tilde{X}_s\right)}
+ \frac{1}{2} Tr\lsz \tilde{a}_s\psz \right) ds \nonumber \\
& \quad + 2q(q-1)\E \int^t_0 \chi_{\{\tau_K \geq s \} } \abs{\tilde{X}_s}^{2(q-2)}
Tr\lsz \tilde{a}_s \lp \tilde{X}_s \otimes \tilde{X}_s \rp \psz ds \nonumber \\
& \leq |x|^{2q} + 2q\tilde{K}_{q-1}\tilde{c}_{q-1}\E \int^t_0 \chi_{\{\tau_K \geq s \} } \lp \abs{\tilde{X}_s}^{2q} + 1 \rp ds \nonumber \\
& = |x|^{2q} + 2q\tilde{K}_{q-1}\tilde{c}_{q-1} (t \wedge \tau_K) + 2q\tilde{K}_{q-1}\tilde{c}_{q-1} \int^t_0 \E \abs{\tilde{X}_{s \wedge \tau_K}}^{2q} ds,
\end{align}
where we have used the joint monotonicity in the form of (\ref{eq:joint dissipativity 2}) and coercivity-type estimate (\ref{eq:coercivity type estimate}).
\noindent By Gronwall lemma it easily follows that
\begin{equation*}
\E \abs{\tilde{X}_{t \wedge \tau_K}}^{2q} \leq \lp |x|^{2q} + 2q\tilde{K}_{q-1}\tilde{c}_{q-1} (t \wedge \tau_K)  \rp e^{2q\tilde{K}_{q-1}\tilde{c}_{q-1} t}
\leq |x|^{2q} e^{2q\tilde{K}_{q-1} \lp  c_{q-1} - \frac{r}{2} \rp t}.
\end{equation*}
The last estimate can be made for $r \geq 2c_{q-1}$. Consequently, the final estimate reads
\begin{equation*}
\E \abs{\tilde{X}_t}^{2q} \leq |x|^{2q} e^{2q\tilde{K}_{q-1} \lp  c_{q-1} - \frac{r}{2} \rp t},
\end{equation*}
and it follows by Fatou. \smallskip
Expressed in terms of the process $X$ we have that for all $t\geq0$ it holds
\begin{equation}\label{eq:state moment estimate final}
\E \lbr e^{-rtq}\lav X_t\pav ^{2q} \rbr \leq |x|^{2q} e^{2q\tilde{K}_{q-1} \lp  c_{q-1} - \frac{r}{2} \rp t}.
\end{equation}
Now it is sufficient to integrate both sides of (\ref{eq:state moment estimate final}) on $[0,+\infty)$.
\smallskip

\noindent Proof of estimate (\ref{eq:state estimate 3}): Fix $q\geq2$ and
observe that $e^{-rt}\lav X_t\pav ^{2q} = e^{rt(q-1)}\abs{\tilde{X}_t}^{2q}$. Therefore, applying It\^o formula to
$e^{rt(q-1)}\abs{\tilde{X}_t}^{2q}$ gives
\begin{align*}
\E & \left[ e^{r(q-1)(t \wedge \tau_K)} \abs{\tilde{X}_{t \wedge \tau_K}}^{2q}\right] = |x|^{2q}
+ 2q\E \int^t_0 \chi_{\{\tau_K \geq s \} } e^{r(q-1)s}\abs{\tilde{X}_s}^{2(q-1)} \times \nonumber \\
& \quad \quad \quad \quad \times \left(\braket{\tilde{X}_s, \tilde{b}(s, \tilde{X}_s)}
+ \frac{1}{2} Tr\lsz \tilde{a}_s\psz \right) ds \nonumber \\
& \quad + 2q(q-1)\E \int^t_0 \chi_{\{\tau_K \geq s \} } e^{r(q-1)s} \abs{\tilde{X}_s}^{2(q-2)}
Tr\lsz \tilde{a}_s \lp \tilde{X}_s \otimes \tilde{X}_s \rp \psz ds \nonumber \\
& \quad + r(q-1)\E \int^t_0 \chi_{\{\tau_K \geq s \} } e^{r(q-1)s} \abs{\tilde{X}_s}^{2q}ds \nonumber \\
& \leq |x|^{2q} + 2q\tilde{K}_{q-1}\tilde{c}_{q-1}\E \int^t_0 \chi_{\{\tau_K \geq s \} } e^{r(q-1)s} \lp \abs{\tilde{X}_s}^{2q} + 1 \rp ds \nonumber \\
& \quad + r(q-1)\E \int^t_0 \chi_{\{\tau_K \geq s \} } e^{r(q-1)s} \abs{\tilde{X}_s}^{2q}ds \nonumber \\
& \leq |x|^{2q} + 2q\tilde{K}_{q-1}\tilde{c}_{q-1}(t \wedge \tau_K)  + \lp 2q\tilde{K}_{q-1}\tilde{c}_{q-1}
+ r(q-1)\rp \E \int^t_0 e^{r(q-1)(s \wedge \tau_K)}\abs{\tilde{X}_{s \wedge \tau_K}}^{2q} ds.
\end{align*}
Then, similarly as before, we obtain
\begin{equation}\label{eq:state estimate 2}
\E \int^{\infty}_0 e^{-rt} |X_t|^{2q} dt \leq C_2|x|^{2q}
,\quad \text{for } r>2q\tilde{K}_{q-1} c_{q-1},
\end{equation}
To conclude the proof observe that once we have obtained the estimates \eqref{eq:state estimate 2}  for $q \ge 2$, the case $q \in [1/2,2)$ easily follows by H\"older inequality.
Note that we have proved even more than stated in the assertion of the theorem. Nevertheless, such generality is not needed for the purposes of proving the SMP.
\end{proof}

\section{Spike variation and variation equations}
The derivation of the variational inequality needed for the formulation of SMP is based on expanding the difference of the functional
$J\lp u^{\varepsilon}(\cdot) \rp - J\lp \bar{u}(\cdot) \rp$ where $\bar{u}(\cdot)$ is an optimal control and $u^{\varepsilon}(\cdot)$
is its appropriate perturbation. Since the control variable is allowed to enter also the diffusion term, the expansion has to be carried out
up to second order due to the time irregularity of the noise. Therefore, two forward variation equations appear in our setting: first
order variaton process $y^{\varepsilon}$ being, in fact, a linearization of the state process, and the second order variation process
$z^{\varepsilon}$ coming from the second order expansion. We also stress that due to the estimation techniques used in the forthcoming proofs, we
often need the polynomial bound of coefficients to be integrable with the weight, which immediately implies that $r$ has to be a priori positive. \bigskip

\noindent Let $\varepsilon > 0$, $E_{\varepsilon} \subset \R_+$ be a set of measure $\varepsilon$ of the form $E_\varepsilon := \left[ t_0, t_0+\varepsilon \right]$, with $t_0\geq 0$ arbitrary but fixed, and $\bar{u}(\cdot)$ an optimal control. Define the spike variation $u^{\varepsilon}(\cdot)$
of $\bar{u}(\cdot)$ by the formula
\[
u^{\varepsilon}_t=
\begin{cases}
\bar{u}_t, & \quad \text{if } t \in \R_+ \setminus E_{\varepsilon}, \\
v, & \quad \text{if } t \in E_{\varepsilon},
\end{cases}
\]
where $v \in U$ is an arbitrary and fixed point. \smallskip

\noindent Let $(\bar{X}(\cdot),\bar{u}(\cdot))$ be a given optimal pair and $(X^{\varepsilon}(\cdot), u^{\varepsilon}(\cdot))$ satisfy the following
perturbed SDE

\begin{equation}\label{SDE:perturbed}
\begin{system}
dX^{\varepsilon}_t = b(X^{\varepsilon}_t,u^{\varepsilon}_t)dt + \sigma(X^{\varepsilon}_t,u^{\varepsilon}_t)dW_t, \\
X^{\varepsilon}_0 = x.
\end{system}
\end{equation}

\noindent Further, following the notation of Yong and Zhou \cite{Yo99}, we denote by $\delta\varphi_t$ the quantity $\varphi(\bar{X}_t,$ $u^{\varepsilon}_t) - \varphi(\bar{X}_t,\bar{u}_t)$, for a generic function $\varphi$. \bigskip

\noindent Now, let us begin studying the first variation equation
\begin{equation}\label{eq:firstVariation}
\begin{system}
dy^{\varepsilon}_t = D_xb(\bar{X}_t,\bar{u}_t)y^{\varepsilon}_tdt + \sum_{j=1}^{d}\bigl[ D_x\sigma^j(\bar{X}_t,\bar{u}_t)y^{\varepsilon}_t + \delta\sigma^j_t\bigr]dW^j_t, \\
y^{\varepsilon}_0 = 0,
\end{system}
\end{equation}

\noindent and the second variation equation

\begin{equation}\label{eq:secondVariation}
\begin{system}
dz^{\varepsilon}_t = \Bigl[ D_xb(\bar{X}_t,\bar{u}_t)z^{\varepsilon}_t + \delta b_t\chi_{E_{\varepsilon}}(t) + \dfrac{1}{2}D^2_x b(\bar{X}_t,\bar{u}_t)(y^{\varepsilon}_t)^2 \Bigr]dt\\
\qquad + \sum_{j=1}^{d}\Bigl[ D_x\sigma^j(\bar{X}_t,\bar{u}_t)z^{\varepsilon}_t + \delta (D_x\sigma^j_t)y^{\varepsilon}_t\chi_{E_{\varepsilon}}(t) + \dfrac{1}{2}D^2_x \sigma^j(\bar{X}_t,\bar{u}_t)(y^{\varepsilon}_t)^2\Bigr]dW^j_t, \\
z^{\varepsilon}_0 = 0,
\end{system}
\end{equation}
where we have adopted the notation

$$
{D^2_x b(\bar{X}_t,\bar{u}_t)(y^{\varepsilon}_t)^2} := \left(
\begin{array}{ccc}
\operatorname{Tr}\bigl[ D^2_x b^1(\bar{X}_t,\bar{u}_t)y^{\varepsilon}_t(y^{\varepsilon}_t)^T \bigr] \\
\vdots \\
\operatorname{Tr}\bigl[ D^2_x b^n(\bar{X}_t,\bar{u}_t)y^{\varepsilon}_t(y^{\varepsilon}_t)^T \bigr]
\end{array}
\right),
$$

$$
{D^2_x \sigma^j(\bar{X}_t,\bar{u}_t)(y^{\varepsilon}_t)^2} := \left(
\begin{array}{ccc}
\operatorname{Tr}\bigl[ D^2_x \sigma^{1j}(\bar{X}_t,\bar{u}_t)y^{\varepsilon}_t (y^{\varepsilon}_t)^T \bigr] \\
\vdots \\
\operatorname{Tr}\bigl[ D^2_x \sigma^{nj}(\bar{X}_t,\bar{u}_t)y^{\varepsilon}_t (y^{\varepsilon}_t)^T \bigr]
\end{array}
\right).
$$

\begin{theorem}\label{th:Variation equations}
Let Hypotheses (H1)-(H4) hold. Then there exist $r_1 \in \R,\ r_2 >0$ such that equation \eqref{eq:firstVariation} (equation \eqref{eq:secondVariation}, resp.) admits a unique solution $y^{\varepsilon} \in L^{2,-r_1}_{\mathcal{F}}(\R_+;\R^n)$ ( $z^{\varepsilon} \in L^{2,-r_2}_{\mathcal{F}}(\R_+;\R^n)$, resp.) for all admissible controls $u(\cdot) \in \mathcal{U}$.
\end{theorem}

\begin{proof}
Note first that \eqref{eq:firstVariation} and \eqref{eq:secondVariation} are linear equations in $y^{\varepsilon}$ and $z^{\varepsilon}$, respectively. The perturbation of the diffusion in \eqref{eq:firstVariation} belongs to $L^{2,-r}_{\mathcal{F}}(\R_+;\R^n)$ for every $r$,  and it is
independent on $y^{\varepsilon}$. Therefore, the joint monotonicity condition (\ref{eq:joint dissipativity 2}) holds and the proof follows by similar arguments as the proof of Theorem \ref{th:1.1}, for $r_1 > 2c_{1/2}$. Concerning the equation for $z^{\varepsilon}$, we have to choose $r_2$ such that $D^2_x b(\bar{X}_t,\bar{u}_t)(y^{\varepsilon}_t)^2,D^2_x \sigma^j(\bar{X}_t,\bar{u}_t)(y^{\varepsilon}_t)^2 \in L^{2,-r_2}_{\mathcal{F}}(\R_+;\R^n)$. Then existence and uniqueness of a solution follow.
\end{proof}

\begin{remark}\label{rem:y}
\normalfont
Let us note that, thanks to the linearity of the equation and due to the structure of the forcing term $\delta \sigma^j$, the solution $y^{\varepsilon}$ to the equation \eqref{eq:firstVariation} is identically zero for times $t \leq t_0$.
\end{remark}

\noindent In the rest of this section, we will often benefit from a general estimate of the solution to a linearized SDE given by the
following Lemma.

\begin{lemma}\label{l:4.1}
Let $Y \in L^{2,-r}_{\Fcal} \lp \R_+ ; \R^n \rp$ be a solution to the following linear SDE
\begin{equation} \label{eq:state eq Lemma 4.1}
\begin{system}
dY_t = \bigl( A_tY_t + \alpha_t \bigr)dt + \sum^d_{j=1} \bigl(B^j_tY_t + \beta^j_t  \bigr)dW^j_t, \nonumber \\
\;\,Y_0  = y_0,
\end{system}
\end{equation}
where $y_0 \in \R^n;$ $A, B^j : \R_+ \times \Omega \rightarrow \R^{n \times n}, \ \alpha, \beta^j : \R_+ \times \Omega \rightarrow \R^n, \ j=1,...,d,$ all are $(\Fcal_t)-$ progressively measurable processes. Let $k \ge 1$, $p>0$ and $c_p \in \R$
such that \smallskip

\begin{enumerate}[1.]
\item $\braket{A_t Y_t,Y_t} + p\sum^d_{j=1}|B^jY_t|^2 \leq c_p|Y_t|^2, \quad \mP \otimes dt-$a.e.

\item $ \int^{\infty}_0  \Bigl[ e^{-\frac{r}{2}t}\bigl( \E |\alpha_t|^{2k} \bigr)^{\frac{1}{2k}} +
 e^{-rt} \Bigl(\E \bigl( \sum^d_{j=1}|\beta^j_t|^2 \bigr)^{k} \Bigr)^{\frac{1}{k}}  \Bigr]dt < +\infty, \quad  1 \leq j \leq d, \text{ and } \,r > 2c_{2k-1}.$
\end{enumerate}

\noindent Then it holds
\begin{equation}\label{eq:l.4.1 sup estimate}
\sup_{t \in \R_+} e^{-rkt}\E\abs{Y_t}^{2k}
\leq K \Biggl[ \E|y_0|^{2k} + \Bigl( \int^{\infty}_0 e^{-\frac{r}{2}t}\bigl(\E |\alpha_t|^{2k}\bigr)^{\frac{1}{2k}}dt \Bigr)^{2k}
+ \Bigl( \int^{\infty}_0 e^{-rt}\Bigl( \E \bigl( \sum^d_{j=1}|\beta^j_t|^{2} \bigr)^{k}  \Bigr)^{\frac{1}{k}} dt \Bigr)^{k}  \Biggr],
\end{equation}
where $K = K(\delta)$, for some appropriate $\delta>0$ and $r > 2c_{2k-1}$.
\end{lemma}

\begin{proof}
The proof will be given for all $B^j, \beta^j$'s bounded. Then the stochastic integral in the proof is a true (centered) martingale. The proof for the unbounded case follows immediately by standard localization and the Fatou lemma.\smallskip

Let $2k, k\geq 2$ and apply the It\^o formula to $e^{-rt}|Y_t|^{2k}$ on $[0,t]$. The case
$k \in [1/2,2)$ follows easily by the H\"{o}lder inequality.
\begin{equation}\label{eq:ito dual Y}
\begin{split}
\E \bigl(e^{-rkt}|Y_t|^{2k} \bigr) &= \E |y_0|^{2k} + 2k\E\int^t_0 e^{-rks} |Y_s|^{2k-2} \braket{A_s Y_s + \alpha_s,Y_s} ds  \\
& \quad + 2k(k-1)\sum^d_{j=1}\E\int^t_0 e^{-rks} |Y_s|^{2k-4} \braket{B^j_s Y_s + \beta^j_s,Y_s}^2 ds   \\
& \quad + k\sum^d_{j=1}\E\int^t_0 e^{-rks} |Y_s|^{2k-2} \braket{B^j_s Y_s + \beta^j_s,B^j_s Y_s + \beta^j_s} ds   \\
& \quad -rk\E\int^t_0 e^{-rks} |Y_s|^{2k}ds  \\
&\leq \E |y_0|^{2k} + 2k \E\int^t_0 e^{-rks} |Y_s|^{2k-2} \Bigl[ \underbrace{\braket{A_s Y_s,Y_s} + (2k-1)\sum^d_{j=1}|B^jY_s|^2 }_{\leq c_{2k-1}|Y_s|^2} \Bigr]ds  \\
& \quad + 2k\E\int^t_0 e^{-rks} |Y_s|^{2k-1}|\alpha_s|ds + 2k(2k-1)\E\int^t_0 e^{-rks} |Y_s|^{2k-2}\Bigl(\sum^d_{j=1}|\beta^j_s|^2 \Bigr)ds  \\
& \quad -rk\E\int^t_0 e^{-rks} |Y_s|^{2k}ds.
\end{split}
\end{equation}

\noindent Now, using H\"{o}lder and the weighted Young inequality $ab \leq \frac{a^p\delta^p}{p} + \frac{b^q}{q\delta^q},\ a,b\geq0, \delta >0$, the remaining terms can be treated as follows

\begin{equation}\label{eq:ito dual Y sup1}
\begin{split}
2k&\E\int^t_0 e^{-rks} |Y_s|^{2k-1}|\alpha_s|ds
\leq 2k\Bigl( \sup_{t \in \R_+} e^{-rkt}\E\left( |Y_t|^{2k-1}\right)^{\frac{2k}{2k-1} }\Bigr)^{\frac{2k-1}{2k}}
\int^{\infty}_0  e^{-\frac{r}{2} t}\bigl(\E|\alpha_t|^{2k} \bigr)^{\frac{1}{2k}}dt \\
& \leq (2k-1)\delta^{\frac{2k}{2k-1}}\Bigl( \sup_{t \in \R_+} e^{-rkt}\E |Y_t|^{2k} \Bigr) +
\frac{1}{\delta^{2k}} \Bigl( \int^{\infty}_0 e^{-\frac{r}{2} t} \bigl( \E|\alpha_t|^{2k} \bigr)^{\frac{1}{2k}}dt \Bigr)^{2k},
\end{split}
\end{equation}
and similarly
\begin{equation}\label{eq:ito dual Y sup2}
\begin{split}
2k&(2k-1)\E\int^t_0 e^{-rks} |Y_s|^{2k-2}\Bigl(\sum^d_{j=1}|\beta^j_s|^2 \Bigr)ds
\leq 2(k-1)(2k-1)\delta^{\frac{2k}{2k-2}}\Bigl( \sup_{t \in \R_+} e^{-rkt}\E |Y_t|^{2k} \Bigr) \\
& \quad + 2(2k-1)\frac{1}{\delta^{k}} \Biggl( \int^{\infty}_0 e^{-rt}\Bigl( \E\sum^d_{j=1}|\beta^j_t|^{2k} \Bigr)^{\frac{1}{k}}dt \Biggr)^{k}.
\end{split}
\end{equation}

\noindent The estimate (\ref{eq:l.4.1 sup estimate}) easily follows by substituting \eqref{eq:ito dual Y sup1} and  \eqref{eq:ito dual Y sup2} into \eqref{eq:ito dual Y}, by taking $\sup_{t\geq 0}$ on both sides
and finally by choosing $\delta>0$ such that $1- (2k-1)\delta^{\frac{2k}{2k-1}} -2(k-1)(2k-1)\delta^{\frac{2k}{2k-2}} >0$ and $r > 2c_{2k-1}$.
\end{proof}

\noindent Before proceeding, let us recall that by $\bar{X}$ and $X^{\varepsilon}$ we mean the solution to \eqref{SDE} in the space $L^{2,-r}_{\Fcal} \lp \R_+ ; \R^n \rp$, for $r >2c_{1/2}$, corresponding to $\bar{u}(\cdot)$ and $u^{\varepsilon}(\cdot)$, respectively. $y^{\varepsilon}$ and $z^{\varepsilon}$ are the solutions to \eqref{eq:firstVariation} and \eqref{eq:secondVariation}, respectively.

\begin{proposition}\label{p:expansion}
Suppose Hypotheses (H1)-(H4) hold and $r >2c_{1/2}$.
Define $\xi^{\varepsilon}_t:= X^{\varepsilon}_t - \bar{X}_t$, $\eta^{\varepsilon}_t:= \xi^{\varepsilon}_t - y^{\varepsilon}_t$ and $\zeta^{\varepsilon}_t:=\xi^{\varepsilon}_t - y^{\varepsilon}_t - z^{\varepsilon}_t, \ t\geq 0$. Then there exist $\rho_1, ...,\rho_5 >0$  such that for $k = 1,2, \ldots$ it holds
\begin{itemize}
\item[(i)] $\sup_{t \in \R_+} e^{-\rho_1 kt}\E\abs{\xi^{\varepsilon}_t}^{2k} = O(\varepsilon^k)$,
\item[(ii)] $\sup_{t \in \R_+} e^{-\rho_2 kt} \E\abs{y^{\varepsilon}_t}^{2k}= O(\varepsilon^k)$,
\item[(iii)] $\sup_{t \in \R_+} e^{-\rho_3 kt} \E\abs{z^{\varepsilon}_t}^{2k}= O(\varepsilon^{2k})$,
\item[(iv)] $\sup_{t \in \R_+} e^{-\rho_4 kt} \E\abs{\eta^{\varepsilon}_t}^{2k} = O(\varepsilon^{2k})$,
\item[(v)] $\sup_{t \in \R_+} e^{-\rho_5 kt} \E\abs{\zeta^{\varepsilon}_t}^{2k} = o(\varepsilon^{2k})$.
\end{itemize}
\end{proposition}

\begin{proof}
See Appendix.
\end{proof}

Before giving a preliminary expansion of the cost, we state the following
\begin{lemma}\label{l.expansion}
If $g \in C^2(\R^n;\R)$ then the following equality holds for every $x, \bar{x} \in \R^n$
\[ g(x) = g(\bar{x}) + \braket{D_xg(\bar{x}), x - \bar{x}} + \int^1_0 \braket{\theta D_x^2 g(\theta\bar{x} + (1-\theta)(x- \bar{x})),x-\bar{x}}d\theta.\]
\end{lemma}

\begin{proposition}\label{p:cost:functional1}
The following expansion holds for the cost functional
\begin{equation}\label{eq:spike_cost1}
\begin{split}
J(u^{\varepsilon}(\cdot)) - J(\bar{u}(\cdot)) &= \E \int_0^{\infty} e^{-rt} \braket{D_xf(\bar{X}_t,\bar{u}_t), y^{\varepsilon}_t + z^{\varepsilon}_t }dt \\
& \quad + \E \int_0^{\infty} e^{-rt}\left[ \frac{1}{2}\braket{D_x^2f(\bar{X}_t,\bar{u}_t) y^{\varepsilon}_t, y^{\varepsilon}_t} + \delta f(\bar{X}_t,\bar{u}_t) \right]dt + o(\varepsilon),
\end{split}
\end{equation}
where the discount factor $r \geq \max_{i=1,...,5}\{ \rho_i \}$ and $\rho_i$ are the individual discount factors from Proposition \ref{p:expansion}.
\end{proposition}

\begin{proof}[Proof of Proposition \ref{p:cost:functional1}]
Thanks to Lemma \ref{l.expansion}, we have
\begin{equation}
\begin{split}
J(u^{\varepsilon}(\cdot)) - J(\bar{u}(\cdot)) &= \E \int_0^{\infty} e^{-rt} \left[ f(X^{\varepsilon}_t, u^{\varepsilon}_t) - f(\bar{X}_t,\bar{u}_t) \right] dt \\
&= \E \int_0^{\infty} e^{-rt} \left[ \delta f(\bar{X}_t,\bar{u}_t) + \braket{D_xf (\bar{X}_t,u^{\varepsilon}_t), \xi_t^{\varepsilon}} \right]dt \\
& \quad + \E \int_0^{\infty} e^{-rt} \int_0^1 \Big<\theta D_x^2 f(\theta \bar{X}_t + (1-\theta)X^{\varepsilon}_t, u^{\varepsilon}_t)\xi^{\varepsilon}_t,\xi^{\varepsilon}_t\Big> d\theta dt. \\
\end{split}
\end{equation}
Finally, Proposition \ref{p:expansion} gives
\begin{equation*}
\begin{split}
J(u^{\varepsilon}(\cdot)) - J(\bar{u}(\cdot)) &= \E \int_0^{\infty} e^{-rt} \Bigg[ \delta f(\bar{X}_t,\bar{u}_t) + \braket{ \delta D_xf(\bar{X}_t,\bar{u}_t), \xi^{\varepsilon}_t}\\
& \qquad + \braket{D_xf(\bar{X}_t,\bar{u}_t), y^{\varepsilon}_t + z^{\varepsilon}_t} + \braket{D_xf(\bar{X}_t,\bar{u}_t), \zeta^{\varepsilon}_t} \\
& \qquad + \int_0^1 \Big<\theta \left( D_x^2 f(\theta \bar{X}_t + (1-\theta)X^{\varepsilon}_t, u^{\varepsilon}_t) - D_x^2 f(\bar{X}_t, u^{\varepsilon}_t) \right) \xi^{\varepsilon}_t ,\xi^{\varepsilon}_t\Big>  d\theta \\
& \qquad + \frac{1}{2} \braket{ \delta D_x^2 f(\bar{X}_t, \bar{u}_t) \xi^{\varepsilon}_t,\xi^{\varepsilon}_t} + \frac{1}{2} \braket{D_x^2 f(\bar{X}_t, \bar{u}_t) y^{\varepsilon}_t, y^{\varepsilon}_t} \\
& \qquad + \frac{1}{2} \braket{D_x^2 f(\bar{X}_t, \bar{u}_t) \eta^{\varepsilon}_t, \xi^{\varepsilon}_t + y^{\varepsilon}_t} \Bigg] dt \\
&= \E \int_0^{\infty} e^{-rt} \braket{D_xf(\bar{X}_t,\bar{u}_t), y^{\varepsilon}_t + z^{\varepsilon}_t }dt \\
& \quad + \E \int_0^{\infty} e^{-rt}\left[ \frac{1}{2}\braket{D_x^2f(\bar{X}_t,\bar{u}_t) y^{\varepsilon}_t, y^{\varepsilon}_t} + \delta f(\bar{X}_t,\bar{u}_t) \right]dt + o(\varepsilon),
\end{split}
\end{equation*}
which completes the proof.
\end{proof}

\section{First adjoint equation}
The first adjoint process naturally arises as a solution to
an appropriate BSDE whose driver can be obtained by differentiating the Hamiltonian function with respect to
the state variable $x$. In some sense, the first adjoint process is dual to the linearized state equation \eqref{eq:firstVariation} and it can have the interpretation of
generalized (in the sense of time-dependent and stochastic) Lagrange multipliers. In the classical setting for BSDEs, the
terminal condition is given a priori. On the contrary, here the BSDE is solved on infinite time horizon and the behaviour at infinity is not known.
Yet the existence and uniqueness result can be derived for processes being in some exponentially-weighted $L^2$ space.

In our case, the first order adjoint equation on infinite time horizon has the following form
\begin{equation}\label{eq.adjoint_first}
dp_t = - \Big[ D_xb(\bar{X}_t,\bar{u}_t)^T p_t  + D_x\sigma(\bar{X}_t,\bar{u}_t)^T q_t -D_xf(\bar{X}_t,\bar{u}_t) - rp_t \Big]dt + q_t dW_t,
\end{equation}
where $D_x\sigma(\bar{X}_t,\bar{u}_t)^T q_t := \sum_{j=1}^d D_x\sigma^j(\bar{X}_t,\bar{u}_t)^T q^j_t \in \R^n$ and
$(\bar{X}_t,\bar{u}_t)$ is an optimal pair.\\

\noindent Let us start the analysis by proving an a priori estimate for the difference of solutions to (\ref{eq.adjoint_first}). To do so, the following estimate will be of a particular interest since it allows to transfer the joint monotonicity property to the BSDE
\begin{equation}\label{eq:estimate.Dsigmaq}
\begin{split}
&\braket{D_x\sigma(\bar{X}_t,\bar{u}_t)^T q_t,p_t} = \sum_{j=1}^d\braket{D_x\sigma^j(\bar{X}_t,\bar{u}_t)^T q^j_t,p_t} = \sum_{j=1}^d\braket{q^j_t,D_x\sigma^j(\bar{X}_t,\bar{u}_t)p_t} \\
&\qquad \qquad \leq \frac{1}{2}\sum_{j=1}^d \abs{q_t^j}^2 + \frac{1}{2}\sum_{j=1}^d \abs{D_x\sigma^j(\bar{X}_t,\bar{u}_t)p_t}^2
\leq \frac{1}{2}||q_t||^2 + \frac{1}{2} \norm{D_x\sigma(\bar{X}_t,\bar{u}_t)p_t}_2^2. \\
\end{split}
\end{equation}

\begin{lemma}\label{l:apriori.backward}
Let $(p^1,q^1)$ and $(p^2,q^2)$ be two solutions to BSDE \eqref{eq.adjoint_first}
belonging to the space $L^{2,-r}_{\mathcal{F}}(\R_+; \R^n \times \R^{n\times d})$, for $r > 2c_{1/2}$,
corresponding to $f = f^1$ and $f = f^2$. Then the following estimate holds true
\begin{equation*}
\begin{split}
\E \int_0^{\infty} e^{-rt}&\left((r-2c_{1/2}-\delta)\abs{p^1_t - p^2_t}^2 + \frac{1}{2}||q^1_t - q^2_t||^2\right) dt \\
&\leq \frac{1}{\delta} \E \int_0^{\infty} e^{-rt}\abs{D_xf^1(\bar{X}_t,\bar{u}_t) - D_xf^2(\bar{X}_t,\bar{u}_t)}^2 dt,
\end{split}
\end{equation*}
\end{lemma}
where $\delta >0$ is sufficiently small.
\begin{proof}
Applying It\^{o} formula to $e^{-rt}\abs{p^1_t - p^2_t}^2$ gives
\begin{equation*}
\begin{split}
&\abs{p_0^1 - p_0^2}^2 + \E \int_0^{\infty}e^{-rt}\left(-r\abs{p^1_t - p^2_t}^2 + ||q^1_t - q^2_t||^2 \right)dt \\
&\quad = \E \int_0^{\infty} 2 e^{-rt} \braket{p^1_t - p^2_t, D_xb(\bar{X}_t,\bar{u}_t)^T (p^1_t - p_t^2)  + D_x\sigma(\bar{X}_t,\bar{u}_t)^T (q^1_t - q^2_t)}dt\\
& \quad \quad +\E \int_0^{\infty}2 e^{-rt} \braket{p^1_t - p^2_t,- r(p^1_t - p^2_t) + D_xf^1(\bar{X}_t,\bar{u}_t) - D_xf^2(\bar{X}_t,\bar{u}_t) }dt \\
&\quad \leq \E \int_0^{\infty} 2 e^{-rt} \left( \braket{p^1_t - p^2_t, D_xb(\bar{X}_t,\bar{u}_t)^T (p^1_t - p_t^2)} + \frac{1}{2} \norm{D_x\sigma(\bar{X}_t,\bar{u}_t)(p^1_t - p^2_t)}_2^2 \right)dt\\
&\quad \quad + \E \int_0^{\infty} 2 e^{-rt}\left(\frac{1}{2}||q^1_t - q^2_t||^2   -r\abs{p^1_t - p^2_t}^2 \right) dt \\
&\quad \quad + \E \int_0^{\infty} 2 e^{-rt}\left( \frac{\delta}{2}\abs{p^1_t - p^2_t}^2 + \frac{1}{2\delta}\abs{D_xf^1(\bar{X}_t,\bar{u}_t) - D_xf^2(\bar{X}_t,\bar{u}_t)}^2 \right)dt\\
&\quad \leq \E \int_0^{\infty} e^{-rt}\left[\left(2c_{1/2} - 2r + \delta \right)\abs{p^1_t - p^2_t}^2 + \frac{1}{2}||q^1_t - q^2_t||^2 \right] dt \\
&\quad \quad +\E \int_0^{\infty} e^{-rt} \frac{1}{\delta}\abs{D_xf^1(\bar{X}_t,\bar{u}_t) - D_xf^2(\bar{X}_t,\bar{u}_t)}^2 dt,
\end{split}
\end{equation*}
where we have used the estimate \eqref{eq:estimate.Dsigmaq}, joint monotonicity assumption (H4) and weighted Young inequality. The conclusion easily follows.
\end{proof}

\noindent Before giving the proof of existence and uniqueness for the first adjoint equation we produce a preliminary result in finite time horizon $T>0$. Let us consider the following equation:
\begin{equation}\label{eq:first.adjoint.finite.horizon}
\begin{system}
dp_t = - \Big[ D_xb(\bar{X}_t,\bar{u}_t)^T p_t  + D_x\sigma(\bar{X}_t,\bar{u}_t)^T q_t -D_xf(\bar{X}_t,\bar{u}_t) - rp_t \Big]dt + q_t dW_t,\\
p_T = 0,
\end{system}
\end{equation}
where $T>0$ is arbitrary but fixed.

\noindent As far as we know, no results in the literature can be used to solve this equation due to the polynomial growth of $D_x\sigma(\bar{X}_t,\bar{u}_t)^T$ in front of $q_t$. In order to produce existence of a solution to such equation we exploit some duality arguments.

\begin{theorem}\label{thm:first adjoint approx}
Under Hypotheses (H1)-(H5) equation \eqref{eq:first.adjoint.finite.horizon} admits a unique solution $(p,q)$ which belongs to $ \in L^{2}_{\mathcal{F}}([0,T]; \R^n)$ $\times L^{2}_{\mathcal{F}}([0,T]; \R^{n\times d})$, for each $T>0$.
\end{theorem}
\begin{proof}
The proof consists in three steps. First, the diffusion term $\sigma$ is approximated so that there exists a solution for each approximating backward equation, by classical results. Second, a duality between these approximate solutions and a properly perturbed first variation equation is established. The last step consists in constructing a unique solution to the original equation (on a finite horizon) by some compactness arguments. \smallskip

\noindent Let us define a sequence of Lipschitz-continuous maps $\sigma^n$ with $\sigma^n(x) \to \sigma(x)$ as $n \to \infty$, for all $x \in \R^n$ so that the joint monotonicity property still holds. An example of such approximation can be given by (see \cite{cerrai2001second})
\begin{equation*}
\sigma^n(x) =
\begin{system}
\sigma(x), \qquad \qquad \qquad \;\text{ if } \abs{x} \le n, \\
\sigma \lp \frac{(n+1)x}{\abs{x}}\rp, \qquad \quad \text{ if } \abs{x} \ge n +1.
\end{system}
\end{equation*}

\noindent Then, for each $n$, the following approximating equation
\begin{equation}\label{eq:first.adjoint.finite.horizon.approx}
\begin{system}
dp^n_t = - \Big[ D_xb(\bar{X}_t,\bar{u}_t)^T p^n_t  + D_x\sigma^n(\bar{X}_t,\bar{u}_t)^T q^n_t -D_xf(\bar{X}_t,\bar{u}_t) - rp^n_t \Big]dt + q^n_t dW_t,\\
p^n_T = 0,
\end{system}
\end{equation}
admits a unique solution $ \lp p^n_t, q^n_t\rp \in L_{\mathcal{F}}^2([0,T];\R^n) \times L_{\mathcal{F}}^2([0,T];\R^{n\times d})$ thanks to Briand et al. \cite{Briand2003109}, Theorem 4.1. \bigskip

\noindent Now, to establish the duality, consider for all $n \in \mathbb{N}$ and for all $\gamma(\cdot) \in L_{\mathcal{F}}^2([0,T];\R^n)$ and
$\eta(\cdot) \in L_{\mathcal{F}}^2([0,T];\R^{n\times d})$ the following perturbed first variation equation

\begin{equation}\label{eq:first variation for first adjoint}
\begin{system}
dy^n_t = \left( D_xb(\bar{X}_t,\bar{u}_t)y^n_t  - ry^n_t + \gamma_t \right)dt  + \left( D_x\sigma^n(\bar{X}_t,\bar{u}_t)y^n_t + \eta_t \right) dW_t,\ t \in(0,T], \\
y^n_0 = 0.
\end{system}
\end{equation}

\noindent Due to Theorem \ref{th:Variation equations} we know that the above equation has a unique solution in $L_{\mathcal{F}}^2([0,T];\R^n)$ for each $n$.
Moreover, using dissipativity it is easy to show that there exists $K >0$ not depending on $n$ such that

\begin{equation}\label{eq:estimate.y^n}
\E \int_0^T \abs{y^n_t}^2 dt \leq K\left[ \E\int_0^T \abs{\gamma_t}^2 dt + \E\int_0^T \norm{\eta_t}^2 dt\right].
\end{equation}

\noindent Next, by applying the It\^o formula to $d\braket{y^n_t,p^n_t}$ we establish the duality relation
\begin{equation}\label{duality}
\E \int_0^T \braket{p_t^n,\gamma_t} dt + \E\int_0^T Tr\left\{ q^n_t \lp\eta_t \rp^T \right\} dt
=  - \E \int_0^T \braket{D_xf(\bar{X}_t,\bar{u}_t), y^n_t} dt.
\end{equation}
Let us define the set $\mathcal{A}:= \lbrace \gamma(\cdot) \in L_{\mathcal{F}}^2([0,T];\R^n): \norm{\gamma}_{L_{\mathcal{F}}^2([0,T];\R^n)} \leq 1 \rbrace$.
If we take $\eta \equiv 0$ in \eqref{duality} we get
\begin{equation}
\begin{split}
\left( \E \int_0^T \abs{p_t^n}^2dt \right)^{1/2} &\leq \sup_{\gamma \in \mathcal{A}}\left[ \E \int_0^T \big|\braket{D_xf(\bar{X}_t,\bar{u}_t), y^n_t} \big| dt \right] \\
&\leq \sup_{\gamma \in \mathcal{A}}\left[ \left( \E \int_0^T \abs{D_xf(\bar{X}_t,\bar{u}_t)}^2dt\right)^{1/2} \left( \E \int_0^T \abs{y^n_t}^2dt\right)^{1/2}\right] \\
&\leq C \left( \E \int_0^T \abs{D_xf(\bar{X}_t,\bar{u}_t)}^2dt\right)^{1/2}.
\end{split}
\end{equation}
If we repeat the same argument with $\gamma \equiv 0$, instead of $\eta$, we finally get
\begin{equation*}
\norm{p^n_t}^2_{L^{2}_{\mathcal{F}}([0,T]; \R^n)} + \norm{q^n_t}^2_{L^{2}_{\mathcal{F}}([0,T]; \R^{n\times d})} \leq C \E \int_0^T \abs{D_xf(\bar{X}_t,\bar{u}_t)}^2 dt.
\end{equation*}

\noindent This way we have obtained a uniform estimate (with respect to $n$) of the $L^2_{\mathcal{F}}-$norm of $(p^n,q^n)$. Hence there exists a subsequence, denoted by abuse of notation again as $(p^n,q^n)$, which converges weakly in $L^2_{\mathcal{F}}([0,T];\R^n) \times L^2_{\mathcal{F}}([0,T];\R^{n\times d})$ to a couple $(p,q)$. Our goal is to verify that $(p,q)$ is the solution to the limit equation

\begin{equation*}\label{eq:first.adjoint.finite.horizon 2}
\begin{system}
dp_t = - \Big[ D_xb(\bar{X}_t,\bar{u}_t)^T p_t  + D_x\sigma(\bar{X}_t,\bar{u}_t)^T q_t -D_xf(\bar{X}_t,\bar{u}_t) - rp_t \Big]dt + q_t dW_t,\\
p_T = 0.
\end{system}
\end{equation*}

\noindent To do so, we note that due to the linearity of the equation, it is enough to prove that each term of the approximating equation weakly converges to the corresponding term in the limit equation. Let us start by studying the term

\begin{equation*}
\begin{split}
\braket{D_x\sigma^n(\bar{X}_t,\bar{u}_t)^Tq^n_t,v} &=  \braket{q^n_t,D_x\sigma^n(\bar{X}_t,\bar{u}_t)v} \\
&= \braket{q^n_t,D_x\sigma(\bar{X}_t,\bar{u}_t)v} + \braket{q^n_t,D_x\sigma^n(\bar{X}_t,\bar{u}_t)v - D_x\sigma^n(\bar{X}_t,\bar{u}_t)v}.\\
\end{split}
\end{equation*}

\noindent As $n \to \infty$, the right hand side converges to
\begin{equation*}
\braket{q_t,D_x\sigma(\bar{X}_t,\bar{u}_t)v} = \braket{D_x\sigma(\bar{X}_t,\bar{u}_t)^Tq_t,v},
\end{equation*}
thanks to the pointwise convergence of the derivative of $\sigma^n$. Indeed, $D_x\sigma^n(x) = D_x\sigma(x)$, if $\abs{x} \le n$, and the derivative is bounded.
Regarding the noise term, let us notice that the map $q \to \int_0^T q_t dW_t$ is linear and continuous, hence weakly continuous. The other terms are easy to treat. \bigskip

\noindent For the uniqueness part it is enough to use a version of Lemma \ref{l:apriori.backward} on finite time horizon. Then we have existence and uniqueness of a solution in finite time horizon and the proof is finished.
\end{proof}

\begin{remark}
\normalfont

\begin{itemize}
\item[]
\item[(a)] The introduction of the term $-r y^n_t$ in (\ref{eq:first variation for first adjoint}) is due to the choice of the scalar product used for
establishing duality. If one considers a scalar product in $L^{2,-r}$ rather than in $L^2$ then the additional term $-r y^n_t$ can be omitted. \smallskip

\item[(b)] An alternative approach to obtain the uniform estimate can be the one following Pardoux \cite{pardoux1999bsdes}. Indeed, applying the It\^o formula gives
\begin{align*}
\E|p^n_t|^2 &= 2 \E \int^T_t \lbr \braket{p^n_s,D_x b(\bar{X}_s,\bar{u}_s)p^n_s}
    + \braket{p^n_s,D_x\sigma^n(\bar{X}_s,\bar{u}_t)^Tq^n_s} + \braket{p^n_s,D_x f(\bar{X}_s,\bar{u}_s)} \rbr ds \nonumber \\
& \quad - 2r\E \int^T_t |p^n_s|^2 ds -  \E \int^T_t \norm{q^n_s}^2 ds,
\end{align*}
which, thanks to the joint monotonicity of $b,\sigma^n$ and weighted Young inequality, produces
\begin{align*}\label{eq:estimate Pardoux-type2}
\E|p^n_t|^2 + &\lp 2r -\varepsilon - 2c_1 \rp \E \int^T_t |p^n_s|^2 ds + \lp 1- \frac{1}{2} \rp \E \int^T_t \norm{q^n_s}^2 ds \nonumber \\
& \leq \frac{1}{\varepsilon}  \E \int^T_t |D_x f(\bar{X}_s,\bar{u}_s)|^2 ds,
\end{align*}
for all $t \in [0,T]$, $\varepsilon >0$ and $r > c_1$. Again, we have a uniform estimate (in $n$) for the left hand side and the relative compactness argument can be applied as before. Note that this approach gives another restriction on $r$ than the one used in the proof. \smallskip

\item[(c)] All the results of this section can be made more general when considering general weighted Young inequality
$ab \leq \frac{p}{2}a^2 + \frac{1}{2p}b^2, \ p>0$ in (\ref{eq:estimate.Dsigmaq}) rather than the usual Young inequality with $p=1$.

\end{itemize}
\end{remark}

\noindent Now we are ready for the following
\begin{theorem}\label{t:adjoint:eq}
Under Hypotheses $(H1)-(H5)$, there exists $r >0$ such that equation \eqref{eq.adjoint_first} admits a unique solution $(p,q)$ which belongs to $L^{2,-r}_{\mathcal{F}}(\R_+; \R^n) \times L^{2,-r}_{\mathcal{F}}(\R_+; \R^{n\times d})$.
\end{theorem}
\begin{proof}
Following Peng and Shi \cite{peng2000infinite}, Theorem 4, define for all $k \in \mathbb{N}$
\[\varphi_t^k := D_xf(\bar{X}_t,\bar{u}_t)\chi_{[0,k]}(t), \qquad t \in \R_+,\]
which converges to $D_xf(\bar{X}_t,\bar{u}_t)$ as $k \to \infty$. We define the solution to the following approximate equation on infinite time
horizon
\begin{equation}
dp^k_t = - \Big[ D_xb(\bar{X}_t,\bar{u}_t)^T p^k_t  + D_x\sigma(\bar{X}_t,\bar{u}_t)^T q^k_t -\varphi_t^k - rp^k_t \Big]dt + q^k_t dW_t,
\end{equation}
as a process solving the following BSDE on a finite time horizon
\begin{equation}
\begin{system}
dp^k_t = - \Big[ D_xb(\bar{X}_t,\bar{u}_t)^T p^k_t  + D_x\sigma(\bar{X}_t,\bar{u}_t)^T q^k_t -\varphi_t^k - rp^k_t \Big]dt + q^k_t dW_t, \\
p_k^k = 0,
\end{system}
\end{equation}
and which is identically zero for $t \in (k,\infty)$. Such solution exists for each $k$ due to Theorem \ref{thm:first adjoint approx}. Using the a priori estimate given in Lemma \ref{l:apriori.backward}, it is easy to see that there exists $r$ such that the sequence of solutions $(p^k_t,q^k_t)$ forms a Cauchy sequence in $L^{2,-r}_{\mathcal{F}}(\R_+; \R^n) \times L^{2,-r}_{\mathcal{F}}(\R_+; \R^{n\times d})$ and that the limiting processes $(p_t,q_t)$ solves \eqref{eq.adjoint_first}. Uniqueness is straightforward due to Lemma \ref{l:apriori.backward}.
\end{proof}

\section{Second adjoint}
The second adjoint equation has the following form

\begin{equation}\label{eq:second:adjoint}
\begin{split}
- dP(t) &= \Big[ D_xb(\bar{X}_t,\bar{u}_t)^T P_t  + P_tD_xb(\bar{X}_t,\bar{u}_t) \\
& \qquad + \sum_{j=1}^d D_x\sigma^j(\bar{X}_t,\bar{u}_t)^T P_tD_x\sigma^j(\bar{X}_t,\bar{u}_t) \\
& \qquad + \sum_{j=1}^d \left( D_x\sigma^j(\bar{X}_t,\bar{u}_t)^T Q^j_t+ Q^j_tD_x\sigma^j(\bar{X}_t,\bar{u}_t) \right) \\
& \qquad  \Big. + D_x^2 H(\bar{X}_t,\bar{u}_t,p_t,q_t) - r_tP_t \Big] dt  -  \sum_{j=1}^d Q_t^j dW^j_t.
\end{split}
\end{equation}

\noindent For a detailed discussion of the role of this equation see e.g. \cite{Yo99}.
We can see that the term $\sum_{j=1}^d D_x\sigma^j(\bar{X}_t,\bar{u}_t)^T P_tD_x\sigma^j(\bar{X}_t,\bar{u}_t)$ destroys the dissipative behaviour of the dynamics in the sense that, in general
\begin{equation}\label{eq:dissipativity P}
\Big< \sum_{j=1}^d D_x\sigma^j(\bar{X}_t,\bar{u}_t)^T P_tD_x\sigma^j(\bar{X}_t,\bar{u}_t),P_t \Big>_2 \nleq c_p \norm{P_t}^2_2.
\end{equation}

\noindent Nontheless, see Remark \ref{rm:sigma lipschitz} for one particular case. The lack of dissipativity prevents us from finding an a priori estimate of the solution. Hence the argument we adopted to solve the first adjoint is no longer helpful. The only information that can be useful to study the process $P_t$ comes from the first variation equation (\ref{eq:firstVariation}). It can be shown that $P_t$ is dual (in some sense explained later) to the process
$Y^{\varepsilon}_t$ defined as $Y^{\varepsilon}_t = y^{\varepsilon}_t(y^{\varepsilon}_t)^T$. It is not difficult to verify that $Y^{\varepsilon}_t$ is a symmetric and positive (semi)definite matrix process. By using It\^o formula it can be also shown that it is a solution to the following (matrix-valued) SDE
\begin{equation}\label{eq:Y.general}
\begin{split}
dY^{\varepsilon}_t &= \Bigl[ D_xb(\bar{X}_t,\bar{u}_t)Y^{\varepsilon}_t + Y^{\varepsilon}_tD_xb(\bar{X}_t,\bar{u}_t)^T \Bigr. \\
& \qquad + \sum_{j=1}^d D_x\sigma^j(\bar{X}_t,\bar{u}_t) Y^{\varepsilon}_t D_x\sigma^j(\bar{X}_t,\bar{u}_t)^T + \Gamma(t)\Bigr] dt \\
& \qquad + \sum_{j=1}^d \Bigl[ D_x\sigma^j(\bar{X}_t,\bar{u}_t) Y^{\varepsilon}_t +
Y^{\varepsilon}_tD_x\sigma^j(\bar{X}_t,\bar{u}_t)^T  + \Lambda^j(t) \Bigr]dW^j_t, \\
Y^{\varepsilon}_0 &= 0,
\end{split}
\end{equation}

\noindent where
\begin{equation}\label{eq:Gamma}
\begin{split}
\Gamma(t) &:= \sum_{j=1}^d \delta \sigma^j(\bar{X}_t,\bar{u}_t) \left(\delta \sigma^j(\bar{X}_t,\bar{u}_t) \right)^T  + \sum_{j=1}^dD_x\sigma^j(\bar{X}_t,\bar{u}_t) y^{\varepsilon}_t \left(\delta\sigma^j(\bar{X}_t,\bar{u}_t)\right)^T \\
&\qquad + \sum_{j=1}^d \delta\sigma^j(\bar{X}_t,\bar{u}_t) (y^{\varepsilon}_t)^T D_x\sigma^j(\bar{X}_t,\bar{u}_t)^T,
\end{split}
\end{equation}
\noindent and
\begin{equation}\label{eq:Lambda}
\begin{split}
\Lambda(t) &= \sum_{j=1}^d \Lambda^j(t) := \sum_{j=1}^d \Bigl[ \delta\sigma^j(\bar{X}_t,\bar{u}_t) (y^{\varepsilon}_t)^T + y^{\varepsilon}_t \left(\delta \sigma^j(\bar{X}_t,\bar{u}_t) \right)^T \Bigr].
\end{split}
\end{equation}

\noindent We also have the following

\begin{proposition}\label{p:estimate_Y}
Under Hypotheses (H1)-(H5), there exists $r \in \R$ such that equation \eqref{eq:Y.general} has a unique solution $Y^{\varepsilon} \in L_{\mathcal{F}}^{2,-r}(\R_+; \R^{n\times n})$ and the following holds

\begin{equation}\label{eq:estimate_Y}
\E \int_0^{\infty}e^{-rt}\norm{Y^{\varepsilon}_t}_2^2 dt \leq K \E \int_0^{\infty}e^{-rt}\norm{\Gamma_t}_2^2 dt + K\sum_{j=1}^d \E \int_0^{\infty}e^{-rt}\norm{\Lambda^j_t}_2^2 dt, \end{equation}
for some $K >0$.
\end{proposition}

\begin{proof}
See Appendix.
\end{proof}

The crucial point here is that Proposition \ref{p:estimate_Y} holds true if and only if $Y^{\varepsilon}_t$ is the solution to equation \eqref{eq:Y.general}, i.e. for $\Gamma$ and $\Lambda$ given by (\ref{eq:Gamma}) and (\ref{eq:Lambda}). For general (nonsymmetric) forcing terms $\Gamma$ and $\Lambda^j \in L_{\mathcal{F}}^{2,-r}(\R_+;\R^{n\times n})$ the corresponding process $Y_t$ can not be decomposed anymore as $y_ty_t^T$ for some process $y_t$. Due to this fact, it is not possible to apply a classical duality argument (as in \cite{tessitore1996existence} or \cite{Yo99}) to extract some information for $P$ and the corresponding BSDE.

\begin{remark}\label{rm:sigma lipschitz}
\normalfont
Note that in the case of $\sigma$ Lipschitz (thus $D_x \sigma$ bounded) it is quite easy to derive the dissipativity of $P$ in sense of
(\ref{eq:dissipativity P}). This particular case can be treated in the same way as in the section on first adjoint equation.
\end{remark}

\subsection{Construction of $P_t$}
Here we propose a different way to construct the process $P$, following ideas of Fuhrman et al. \cite{fuhrman2013stochastic}. More precisely, will show that there exists $r>0$ and a well defined matrix-valued process $P$ such that the following duality relation holds
\begin{equation}\label{eq:final.relation}
\begin{split}
\E\int_0^\infty &e^{-rt} \opn{Tr} \Big[ D_x^2 H(\bar{X}_t,\bar{u}_t,p_t,q_t)Y^{\varepsilon}_t\Big] dt \\
&= \sum_{j=1}^d \E\int_0^\infty e^{-rt} \Big<P_t \delta \sigma^j(\bar{X}_t,\bar{u}_t),\delta \sigma^j(\bar{X}_t,\bar{u}_t)\Big>dt  + o(\varepsilon).
\end{split}
\end{equation}
Once we have this relation, it is easy to prove the stochastic maximum principle using usual arguments. The strategy to do so will be the following. \\
\newline
{\bf{Dual identity satisfied by $P$:} }
For $t \ge 0$ and an arbitrary vector $\gamma \in \R^n$, let us consider the following SDE

\begin{equation}\label{eq:firstVariation duality for second adjoint}
\begin{system}
dy^{t,\gamma}_s = D_xb(\bar{X}_s,\bar{u}_s)y^{t,\gamma}_s ds + \sum_{j=1}^{d} D_x\sigma^j(\bar{X}_s,\bar{u}_s)y^{t,\gamma}_s dW^j_s,\ s \geq t,\\
\;\; y^{t,\gamma}_t = \gamma.
\end{system}
\end{equation}
By repeating the arguments by Yong and Zhou \cite{Yo99}, Chapter 3, the SDE for the product
$y^{t,\eta}_s\lp y^{t,\gamma}_s \rp ^T$ is of the form (with the notation $A_t \equiv D_xb(\bar{X}_t,\bar{u}_t)$ and
$B^j_t \equiv D_x\sigma^j(\bar{X}_t,\bar{u}_t)$)

\begin{align}\label{eq:duality for second adjoint ito for Y}
d \lp y^{t,\eta}_s\lp y^{t,\gamma}_s \rp ^T \rp &= \left[ A_s y^{t,\eta}_s\lp y^{t,\gamma}_s \rp ^T  + y^{t,\eta}_s\lp y^{t,\gamma}_s \rp ^T A^T_s
+ \sum_{j=1}^{d} B^j_s  y^{t,\eta}_s\lp y^{t,\gamma}_s \rp ^T  \lp B^j_s \rp^T \right]  ds \nonumber \\
& \qquad + \left[  \sum_{j=1}^{d} B^j_s  y^{t,\eta}_s \lp y^{t,\gamma}_s \rp ^T  +  y^{t,\eta}_s\lp y^{t,\gamma}_s \rp ^T \lp B^j_s \rp^T  \right] dW^j_s.
\end{align}

\noindent Suppose for a moment that we are able to find a solution to equation \eqref{eq:second:adjoint} in
$L^{2,-r}_{\mathcal{F}}\lp \R_+; \mathcal{S}^n\rp \times  (L^{2,-r}_{\mathcal{F}}\lp \R_+; \mathcal{S}^n\rp)^d$ for some $r>0$.
Noting that $\braket{P_s y^{t,\eta}_s ,y^{t,\gamma}_s} = \opn{Tr} \lbrace P_s y^{t,\eta}_s(y^{t,\gamma}_s)^T \rbrace$ and using equation \eqref{eq:duality for second adjoint ito for Y}, it follows by the It\^o formula that for all $[t,T]$ we have $\mP$-almost surely
\begin{equation}
\begin{split}
e^{-rt} \braket{P_t \eta , \gamma} &= \E^{\Fcal_t} \lbr e^{-rt} \braket{P_t y^{t,\eta}_t ,y^{t,\gamma}_t} \rbr \\
&= \E^{\Fcal_t} \lbr e^{-rT} \braket{P_T y^{t,\eta}_T ,y^{t,\gamma}_T} \rbr + \E^{\Fcal_t} \int_t^T e^{-rs} \braket{D_x^2 H(s) y^{t,\eta}_s ,y^{t,\gamma}_s} ds,
\end{split}
\end{equation}
where we have used the notation $D_x^2 H(t):= D_x^2 H(\bar{X}_t,\bar{u}_t,p_t,q_t)$ for the forcing term in the equation for $P$. Since the processes $P(\cdot)$ and $y^{t,\eta}_{\cdot} (y^{t,\gamma}_{\cdot})^T$ are assumed to be in
some appropriate exponentially-weighted spaces, there has to be a sequence of times $(T_n)_{n\geq 1}$ with $T_n \nearrow +\infty$ as $n \rightarrow +\infty$ such that $\mP$-almost surely
\begin{equation}\label{eq:duality for second adjoint vanishing term}
\lim_{n \rightarrow +\infty}\E^{\Fcal_t} \lbr e^{-rT_n} \braket{P_{T_n} y^{t,\eta}_{T_n} ,y^{t,\gamma}_{T_n}} \rbr = 0.
\end{equation}

\noindent Passing to the limit along the above sequence $(T_n)_{n\geq 1}$ produces the following formal relation
\begin{equation}\label{eq:duality for second adjoint duality}
\braket{P_t \eta , \gamma} = \E^{\Fcal_t} \int_t^\infty e^{-r(s-t)} \braket{D_x^2 H(s) y^{t,\eta}_s ,y^{t,\gamma}_s} ds,
\end{equation}
which can be used to define the process $P_t$. Our aim is to show that the right hand side of  \eqref{eq:duality for second adjoint duality} actually defines a continuous bilinear form that can be used to prove \eqref{eq:final.relation} without any reference to the second adjoint BSDE. \bigskip

\noindent {\bf{Existence of $P$:} }
The following estimates on $(y^{t,\eta}_s)_{s\geq t}$ are crucial to prove continuity of the bilinear form.

\begin{proposition}\label{p.tecn.P1}
Let $\eta \in \R^n$ and assume that Hypotheses (H1)-(H4) hold. Then there is a unique solution
$(y^{t,\eta}_s)_{s\geq t}\in L_{\mathcal{F}}^{2,-r}(\R_+;\R^n)$ to the equation \eqref{eq:firstVariation duality for second adjoint} for some $r$.
Moreover, there exists a constant $C >0$ such that for $t \ge 0$ and $s\geq t$
\begin{equation}\label{eq:estimate_y_gamma}
\sup_{s\geq t}\E^{\mathcal{F}_t}\lbr e^{-rt}\abs{y^{t,\eta}_s}^4 \rbr \leq C\abs{\eta}^4, \qquad \mP-\text{a.s.}
\end{equation}
and for all $h>0$, $0 \le t \le t+h$ and $s\geq t+h$
\begin{equation}\label{eq:continuity:y:initial_data}
e^{-rs}\E\abs{y^{t+h,\eta}_s - y^{t,\eta}_s}^4 \leq Ch.
\end{equation}
\end{proposition}

\begin{proof}
Let us choose $r > 2c_{1/2}$. The existence follows immediately by Theorem \ref{th:Variation equations} and the proof of \eqref{eq:estimate_y_gamma} it is a easy consequence of Lemma \ref{l:4.1} with the additional requirement
$r> 2\max\{c_{1/2}, c_3 \}$. To prove the continuity property \eqref{eq:continuity:y:initial_data} let us denote $z_s = y_s^{t+h,\eta} - y_s^{t,\eta}$ then, for $s \ge t+h$, we have by the It\^o formula
\begin{equation*}
\begin{split}
e^{-rs}\E\abs{z_s}^4 &= e^{-r(t+h)}\E\abs{\eta - y^{t,\eta}_{t+h}}^4 -r\E \int_{t+h}^s e^{-r\tau}\abs{ z_\tau}^4d\tau \\
& \quad + \E\int_{t+h}^s e^{-r\tau} \abs{z_\tau}^2\braket{D_xb(\bar{X}_{\tau},\bar{u}_{\tau})z_\tau, z_\tau}d\tau \\
& \quad + \sum_{j=1}^d \E\int_{t+h}^s e^{-r\tau} \braket{D_x\sigma^j(\bar{X}_{\tau},\bar{u}_{\tau})z_\tau,z_\tau}^2 d\tau \\
& \quad + \sum_{j=1}^d \E\int_{t+h}^s e^{-r\tau} \abs{z_\tau}^2\braket{D_x\sigma^j(\bar{X}_{\tau},\bar{u}_{\tau})z_\tau,z_\tau} d\tau. \\
\end{split}
\end{equation*}
Using the same estimate of Lemma \ref{l:4.1} we end up with
\begin{equation}
e^{-rs}\E\abs{y_s^{t+h,\eta} - y_s^{t,\eta}}^4 \leq K e^{-r(t+h)}\E\abs{\eta - y^{t,\eta}_{t+h}}^4,
\end{equation}
which we can control in the following form
\begin{equation*}
\begin{split}
\E\abs{\eta - y^{t,\eta}_{t+h}}^4 &= \E\Bigg| \int_t^{t+h}D_xb(\bar{X}_{\tau},\bar{u}_{\tau})y^{t,\eta}_\tau d\tau + \sum_{j=1}^d \int_t^{t+h}D_x\sigma^j(\bar{X}_{\tau},\bar{u}_{\tau})y^{t,\eta}_\tau dW^j_{\tau}\Bigg|^4 \\
&\leq C\E\int_t^{t+h}\abs{D_xb(\bar{X}_{\tau},\bar{u}_{\tau})y^{t,\eta}_\tau}^4 d\tau + \sum_{j=1}^d \int_t^{t+h}\abs{D_x\sigma^j(\bar{X}_{\tau},\bar{u}_{\tau})y^{t,\eta}_\tau }^4 d\tau.
\end{split}
\end{equation*}
Now, using H\"older inequality and again Lemma \ref{l:4.1} for the first term we obtain
\begin{equation}
\begin{split}
\E&\int_t^{t+h}\abs{D_xb(\bar{X}_{\tau},\bar{u}_{\tau})y^{t,\eta}_\tau}^4 d\tau  \\
& \qquad \leq  \Big(  \E \int_t^{t+h} e^{r\tau} e^{-r\tau} || D_xb(\bar{X}_{\tau},\bar{u}_{\tau})||^8 d\tau \Big)^{\frac{1}{2}}
\Big( \E \int_t^{t+h} e^{r\tau} e^{-r\tau} |y^{t,\eta}_\tau|^8 d\tau\Big)^{\frac{1}{2}} \\
& \qquad \leq h \Big( \sup_{\tau \in [t, t+h]}\left( e^{-r\tau} \E \abs{y_\tau^{t,\eta}}^8\right) \Big)^{\frac{1}{2}}
\Big( \sup_{\tau \in [t, t+h]}\left( e^{-r\tau} \E ||D_xb(\bar{X}_{\tau},\bar{u}_{\tau})|| ^8\right) \Big)^{\frac{1}{2}}  \\
& \qquad \leq Ch.
\end{split}
\end{equation}
The $D_x\sigma$ term can be treated in the same way and the conclusion follows.
\end{proof}

\begin{proposition}\label{p:def:P_t}
Let Hypotheses (H1)-(H5) hold and $\gamma$, $\eta \in \R^n$. Then there exist $r >0$ and a progressive process $(P_t)_{t\ge 0}$ with values in $\mathcal{S}^n$ such that for all $t \geq 0$ it holds
\begin{equation}\label{eq:final_duality_second_adjoint}
\braket{P_t \eta , \gamma} = \E^{\Fcal_t} \int_t^\infty e^{-r(s-t)} \braket{D_x^2 H(s) y^{t,\eta}_s ,y^{t,\gamma}_s} ds, \qquad \mP-\text{a.s.}
\end{equation}
Moreover, $\sup_{t\ge 0} \E\norm{P_t}^2 < \infty$ and for $\varepsilon \searrow 0$ we have that
\begin{equation}\label{eq:continuity:P_t}
\E\abs{\braket{(P_{t+\varepsilon}- P_t)\gamma,\eta}} \to 0.
\end{equation}
\end{proposition}

\begin{proof}
For $\gamma$ and $\eta \in \R^n$ fixed, let us define $\braket{P_t\gamma,\eta}$ by the formula given in the statement. To do so we have chosen an arbitrary version of the conditional expectation. To construct the process $P_t$ we have to prove that the map $(\gamma,\eta) \mapsto \braket{P_t\gamma, \eta}$ is a continuous bilinear form. Note that
\begin{equation}\label{eq:P bound estimate}
\begin{split}
\Big|&\E^{\mathcal{F}_t}\int_t^\infty e^{-r(s-t)} \braket{D_x^2 H(s) y^{t,\eta}_s ,y^{t,\gamma}_s} ds \Big| \\
&\le \E^{\mathcal{F}_t}\int_t^\infty e^{-r(s-t)} \abs{D_x^2 H(s)}\abs{y^{t,\eta}_s}\abs{y^{t,\gamma}_s}ds \\
&\leq C\int_t^\infty \left( e^{-r(s-t)}\E^{\mathcal{F}_t}\abs{D_x^2 H(s)}^p \right)^{1/p} \left( e^{-r(s-t)}\E^{\mathcal{F}_t}\abs{y^{t,\eta}_s}^{2q} \right)^{\frac{1}{2q}}
 \left( e^{-r(s-t)}\E^{\mathcal{F}_t}\abs{y^{t,\gamma}_s}^{2q} \right)^{\frac{1}{2q}}ds \\
&\leq C\abs{\eta}\abs{\gamma} \left( \int_t^\infty e^{-r(s-t)}\E^{\mathcal{F}_t}\abs{D_x^2 H(s)}^p ds \right)^{1/p},
\end{split}
\end{equation}
where we used conditional H\"older inequality with $p \in (1,2), q = \frac{p}{p-1} >2$ and estimate \eqref{eq:estimate_y_gamma}. $r >0$ can be chosen such that $\int_t^\infty e^{-rs}\E^{\mathcal{F}_t}\abs{D_x^2 H(s)}^p ds < \infty$. This can be seen from the definition of the Hamiltonian, the estimates on first adjoint processes and the polynomial growth of the coefficients.
Further, there exists a set $N$  such  that $\mP(N) = 0$ and for $\omega \notin N$,
\begin{equation*}
\abs{\braket{P_t(\omega)\eta,\gamma}} \le C\abs{\eta}\abs{\gamma}.
\end{equation*}
If we set $P_t(\omega) = 0$ for $\omega \in N$, by now we have constructed an adapted process $P_t$ which satisfies equation \eqref{eq:final_duality_second_adjoint}. The symmetry of the process $P$ is obtained easily by symmetry of $D^2_x H(s)$.

To construct a progressive modification of $P_t$, it is sufficient to prove that the map
$(\omega,t) \mapsto P_t(\omega)$ is $\mathcal{F} \otimes \mathcal{B}(\R_+) \setminus \mathcal{B}(\R^{n \times n})-$measurable
(i.e. it is a jointly measurable process).
Here, $\mathcal{B}(\R^{n \times n})$ stands for a Borel $\sigma-$field induced by the norm $||\cdot||_2$ on $\R^{n \times n}$.
If we prove that $P$ is an $(\mathcal{F}_t)-$adapted and jointly measurable process then there is an $(\mathcal{F}_t)-$progressive
version of $P$. For a recent and elegant proof of this fact see \cite{ECP2548}.
Concerning joint measurability of $P$, its proof is given in \cite{fuhrman2013stochastic}. In that paper, in fact, even the existence of a progressive version in infinite dimensional setting is provided without any reference to the classical result. \bigskip

\noindent To show that \eqref{eq:continuity:P_t} holds, let us write
\begin{equation*}
\begin{split}
\braket{(P_{t+\varepsilon} - P_t)\eta,\gamma}&= \left( \E^{\mathcal{F}_{t+\varepsilon}}  - \E^{\mathcal{F}_t} \right) \int_t^\infty e^{-r(s-t)}\braket{D_x^2 H(s) y^{t,\eta}_s ,y^{t,\gamma}_s} ds\\
& \quad -\E^{\mathcal{F}_{t+\varepsilon}} \int_t^{t+\varepsilon}e^{-r(s-t)}\braket{D_x^2 H(s) y^{t,\eta}_s ,y^{t,\gamma}_s} ds\\
& \quad + \E^{\mathcal{F}_{t+\varepsilon}} \int_{t+\varepsilon}^{\infty} e^{-r(s-t-\varepsilon)}\left( \braket{D_x^2 H(s) y^{t+\varepsilon,\eta}_s ,y^{t+\varepsilon,\gamma}_s} - \braket{D_x^2 H(s) y^{t,\eta}_s ,y^{t,\gamma}_s} \right)ds \\
& \quad +\E^{\mathcal{F}_{t+\varepsilon}} \int_{t+\varepsilon}^{\infty} \left( e^{-r(s-t-\varepsilon)} - e^{-r(s-t)}\right)  \braket{D_x^2 H(s) y^{t,\eta}_s ,y^{t,\gamma}_s}ds.
\end{split}
\end{equation*}
The first summand on the right hand side goes to zero in $L^1(\Omega)$ as $\varepsilon \searrow 0$ thanks to the L\'evy downward martingale convergence theorem (note that we have UC filtration $(\mathcal{F}_t)_{t \geq 0}$), the second one and the last one tend to zero in $L^1(\Omega)$ by dominated convergence theorem. Regarding the third term the result easily follows by using \eqref{eq:continuity:y:initial_data}. Indeed we can rewrite it as follows
\begin{equation}
\begin{split}
&\E^{\mathcal{F}_{t+\varepsilon}} \int_{t+\varepsilon}^{\infty} e^{-r(s-t-\varepsilon)}\braket{D_x^2 H(s) \lp y^{t+\varepsilon,\eta}_s - y^{t,\eta}_s \rp ,y^{t+\varepsilon,\gamma}_s}ds\\
&+ \E^{\mathcal{F}_{t+\varepsilon}} \int_{t+\varepsilon}^{\infty} e^{-r(s-t-\varepsilon)}\braket{D_x^2 H(s) y^{t,\eta}_s ,y^{t+\varepsilon,\gamma}_s - y^{t,\gamma}_s}ds = I_1 + I_2.
\end{split}
\end{equation}
Using H\"older inequality with $p \in (1,2), q = \frac{p}{p-1} >2$, the first addendum $I_1$ can be estimate by
\begin{equation}
\begin{split}
I_1 \leq e^{t+\varepsilon}&\int_{t+\varepsilon}^{\infty} \left( e^{-rs}\E^{\mathcal{F}_{t+\varepsilon}} \abs{D_x^2 H(s)}^p\right)^{1/p}\left( e^{-rs}\E^{\mathcal{F}_{t+\varepsilon}} \abs{y^{t+\varepsilon,\eta}_s - y^{t,\eta}_s}^{2q}\right)^{\frac{1}{2q}} \\
& \qquad \cdot \left( e^{-rs}\E^{\mathcal{F}_{t+\varepsilon}} \abs{y^{t+\varepsilon,\gamma}_s}^{2q}\right)^{\frac{1}{2q}} ds.
\end{split}
\end{equation}
Repeating the same estimate for the second addendum  $I_2$, using Lemma \ref{l:4.1} and \eqref{eq:continuity:y:initial_data} we get the required result.
\end{proof}

\begin{remark}\label{rem:def_P_RV}
\normalfont
If $F,G$ are random variables in $L^2(\Omega)$ measurable with respect to $\mathcal{F}_t$ then it is true that
\[ \braket{P_tF,G} = \E^{\Fcal_t} \int_t^\infty e^{-r(s-t)} \Big<D_x^2 H(s) y^{t,F}_s ,y^{t,G}_s \Big> ds, \qquad \mP-\text{a.s.} \]
The proof follows by applying similar procedure as in Peng and Shi, \cite{peng2000infinite}, Theorem 13.
\end{remark}

\begin{proposition}\label{p:last:estimtes}
Let $\lp y_t^\varepsilon\rp_{t \geq 0}$ be a solution to the first variation equation (\ref{eq:firstVariation}).
Then there exists $r>0$ such that the following relations hold true.
\begin{equation*}
\begin{split}
&i) \quad  e^{-r\left(t_0+\varepsilon\right)}\E \big<\left( P_{t_0+\varepsilon} - P_{t_0}\right)y_{t_0+\varepsilon}^\varepsilon,y_{t_0+\varepsilon}^\varepsilon\big> = o(\varepsilon), \\
&ii) \quad e^{-r\left(t_0+\varepsilon\right)}\E\braket{P_{t_0}y_{t_0+\varepsilon}^\varepsilon,y_{t_0+\varepsilon}^\varepsilon} \\
&\qquad = \sum_{j=1}^d \E\int_0^\infty e^{-rs} \big<P_s\delta \sigma^j(\bar{X}_s,\bar{u}_s),\delta \sigma^j(\bar{X}_s,\bar{u}_s)\big>ds + o(\varepsilon).
\end{split}
\end{equation*}
\end{proposition}

\begin{proof}
\textbf{(i)} From Proposition \ref{p:expansion}-(ii) we know that there exists $r$ such that
\begin{equation}\label{eq:prop:exp_rewritten}
\left( e^{-r(t_0+\varepsilon)}\E\abs{\varepsilon^{1/2}y_{t_0+\varepsilon}^\varepsilon}^{2k} \right)^{1/2k} \leq C , \qquad k \geq 1,
\end{equation}
and by the Markov inequality, for every $\delta >0$ we have that
\begin{equation*}
\mP\left( \abs{ \varepsilon^{1/2}y_{t_0+\varepsilon}^\varepsilon} >C\delta^{-1/4} \right) \leq e^{r(t_0+\varepsilon)}\delta.
\end{equation*}
If we denote $\Omega_{\delta,\varepsilon}$ the event $\lbrace\varepsilon^{-1/2}y_{t_0+\varepsilon}^\varepsilon \in B_{C\delta^{-1/4}}\rbrace$, where $B_{C\delta^{-1/4}}$ is the centred ball with radius $\delta^{-1/4}$, then it holds
\begin{equation}\label{eq:prob_omega}
\mP(\Omega^c_{\delta,\varepsilon}) \leq e^{r(t_0+\varepsilon)}\delta.
\end{equation}
Now we rewrite \textbf{(i)} in the following form
\begin{equation*}
\begin{split}
&e^{-r\left(t_0+\varepsilon\right)}\E \big< \left( P_{t_0+\varepsilon} - P_{t_0}\right)\varepsilon^{-1/2} y_{t_0+\varepsilon}^\varepsilon,\varepsilon^{-1/2}y_{t_0+\varepsilon}^\varepsilon \big>\\
 &= e^{-r\left(t_0+\varepsilon\right)}\E\left( \big<\left( P_{t_0+\varepsilon} - P_{t_0}\right)\varepsilon^{-1/2} y_{t_0+\varepsilon}^\varepsilon,\varepsilon^{-1/2}y_{t_0+\varepsilon}^\varepsilon\big> 1_{\Omega_{\delta,\varepsilon}^c} \right) \\
 & \quad + e^{-r\left(t_0+\varepsilon\right)}\E\left( \big<\left( P_{t_0+\varepsilon} - P_{t_0}\right)\varepsilon^{-1/2} y_{t_0+\varepsilon}^\varepsilon,\varepsilon^{-1/2}y_{t_0+\varepsilon}^\varepsilon\big> 1_{\Omega_{\delta,\varepsilon}} \right) \\
 &=: A_1^\varepsilon + A_2^\varepsilon.
\end{split}
\end{equation*} 	
The first term can be easily treated by the H\"older inequality, Proposition \ref{p:def:P_t} and estimates \eqref{eq:prop:exp_rewritten}, \eqref{eq:prob_omega}, respectively. Hence, the estimate reads
\begin{equation}
\begin{split}
\abs{A_1^{\varepsilon}} &\leq \left(e^{-r\left(t_0+\varepsilon\right)} \E\norm{P_{t_0+\varepsilon} - P_{t_0}}^2 \right)^{1/2} \\
&\quad \cdot\left( e^{-r\left(t_0+\varepsilon\right)}\E\abs{\varepsilon^{-1/2} y_{t_0+\varepsilon}^\varepsilon}^8 \right)^{1/4}\left( e^{-r\left(t_0+\varepsilon\right)}\mP\left(\Omega_{\delta,\varepsilon}^c \right)\right)^{1/4} \\
&\leq C\delta^{1/4}.
\end{split}
\end{equation}

\noindent Regarding the second term, we have that
\[ \abs{A_2^\varepsilon} \leq e^{-r\left(t_0+\varepsilon\right)}\E \left[ \sup_{x \in B_{C\delta^{-1/4}}} \big|\braket{\left( P_{t_0+\varepsilon} - P_{t_0}\right)x,x} 1_{\Omega_{\delta,\varepsilon}}\big| \right]. \]
Since $B_{C\delta^{-1/4}}$ is compact, there exist $N_\delta$ open balls $B_{x_i,\delta}$ which cover it. Moreover, for all $x \in B_{C\delta^{-1/4}}$ we can choose $i$ such that $\abs{x-x_i} \le \delta$. Then
\begin{equation}
\begin{split}
\braket{\left( P_{t_0+\varepsilon} - P_{t_0}\right)x,x} &= \braket{\left( P_{t_0+\varepsilon} - P_{t_0}\right)x_i,x_i} - \braket{\left( P_{t_0+\varepsilon} - P_{t_0}\right)(x-x_i),(x-x_i)}\\
& \quad + 2\braket{\left( P_{t_0+\varepsilon} - P_{t_0}\right)x,(x-x_i)} \\
&= \braket{\left( P_{t_0+\varepsilon} - P_{t_0}\right)x_i,x_i} + \norm{P_{t_0+\varepsilon} - P_{t_0}}_{\infty} \delta^2\\
& \quad + 2\norm{P_{t_0+\varepsilon} - P_{t_0}}_{\infty}\abs{x}\delta,
\end{split}
\end{equation}
where, for a generic matrix $T \in \R^{n\times n}$, we have used $\norm{T}_{\infty} := \sup \lbrace \abs{\braket{Tx,y}}: x,y \in \R^n, \abs{x} \leq 1, \abs{y} \leq 1 \rbrace$. Taking supremum  and expectation we obtain
\begin{equation}
\abs{A_2^\varepsilon} \leq \sum_{i=1}^{N_\delta}\E\abs{\braket{\left( P_{t_0+\varepsilon} - P_{t_0}\right)x_i,x_i}} + C\left( \delta^2 + \delta^{3/4}\right).
\end{equation}
If we let $\varepsilon \to 0$ and use \eqref{eq:continuity:P_t} it follows that
\[ \limsup_{\varepsilon\to 0}\abs{A_2^\varepsilon} \leq  C\left( \delta^2 + \delta^{3/4}\right),\]
hence $\abs{A_1^\varepsilon} + \abs{A_2^\varepsilon} \to 0$, when $\delta \to 0$ and the proof of \textbf{(i)} is finished.\\
\newline
\textbf{(ii)} Let us rewrite $e^{-r\left(t_0+\varepsilon\right)}\E\braket{P_{t_0}y_{t_0+\varepsilon}^\varepsilon,y_{t_0+\varepsilon}^\varepsilon}$ in the following form
\begin{equation*}
e^{-r\left(t_0+\varepsilon\right)}\E\braket{P_{t_0}y_{t_0+\varepsilon}^\varepsilon,y_{t_0+\varepsilon}^\varepsilon}
= \E \left[ \opn{Tr}\left\{ P_{t_0}\left(e^{-\frac{r}{2}\left(t_0+\varepsilon\right)}y_{t_0+\varepsilon}^\varepsilon\right) \left(e^{-\frac{r}{2}\left(t_0+\varepsilon\right)}y_{t_0+\varepsilon}^\varepsilon\right)^T  \right\}  \right].
\end{equation*}
Thanks to the It\^o formula and equation \eqref{eq:Y.general}, we obtain
\begin{equation}
\begin{split}
d\left(e^{-rs}Y^\varepsilon_s\right) &= e^{-rs}\Bigl[ -rY^\varepsilon_s + D_xb(\bar{X}_s,\bar{u}_s)Y^{\varepsilon}_s + Y^{\varepsilon}_s D_xb(\bar{X}_s,\bar{u}_s)^T \Bigr. \\
& \quad + \sum_{j=1}^d D_x\sigma^j(\bar{X}_s,\bar{u}_s) Y^{\varepsilon}_s D_x\sigma^j(\bar{X}_s,\bar{u}_s)^T + \Gamma(s)\Bigr] ds \\
& \quad + \sum_{j=1}^d e^{-rs}\Bigl[ D_x\sigma^j(\bar{X}_s,\bar{u}_s) Y^{\varepsilon}_s +
Y^{\varepsilon}_s D_x\sigma^j(\bar{X}_s,\bar{u}_s)^T  + \Lambda^j(s) \Bigr] dW^j_s,
\end{split}
\end{equation}
where we have used the notation $Y^\varepsilon_s = y^\varepsilon_s\left( y^\varepsilon_s \right)^T$ and $\Gamma(s)$, $\Lambda^j(s)$ are as in \eqref{eq:Gamma}, \eqref{eq:Lambda}. Now, by taking conditional expectation with respect to $\mathcal{F}_{t_0}$ and rewriting the equation in integral form from $t_0$ to $s$
(recall Remark \ref{rem:y}) it follows that
\begin{equation*}
\begin{split}
\E^{\mathcal{F}_{t_0}}\left(e^{-rs}Y^\varepsilon_s\right) &= \E^{\mathcal{F}_{t_0}}\int_{t_0}^s e^{-r\tau}\Bigl[ -rY^\varepsilon_\tau + D_xb(\bar{X}_\tau,\bar{u}_\tau)Y^{\varepsilon}_\tau + Y^{\varepsilon}_\tau D_xb(\bar{X}_\tau,\bar{u}_\tau)^T \Bigr]d\tau \\
&\quad + \sum_{j=1}^d \E^{\mathcal{F}_{t_0}}\int_{t_0}^s \Bigr[ D_x\sigma^j(\bar{X}_\tau,\bar{u}_\tau) Y^{\varepsilon}_\tau D_x\sigma^j(\bar{X}_\tau,\bar{u}_\tau)^T + \Gamma(\tau) \Bigr]d\tau.
\end{split}
\end{equation*}	
Hence, taking into account the definition of $\Gamma$ in (\ref{eq:Gamma}), multiplying by $P_{t_0}$, setting $s=t_0+\varepsilon$ and taking expectation, we arrive at
\begin{equation*}
\begin{split}
&e^{-r\left(t_0+\varepsilon\right)}\E\braket{P_{t_0}y_{t_0+\varepsilon}^\varepsilon,y_{t_0+\varepsilon}^\varepsilon} \\
&= \int_{t_0}^{t_0+\varepsilon} e^{-r\tau} \E \left[ \opn{Tr} \left\{  P_{t_0}\left( -rY^\varepsilon_\tau + D_xb(\bar{X}_\tau,\bar{u}_\tau)Y^{\varepsilon}_\tau + Y^{\varepsilon}_\tau D_xb(\bar{X}_\tau,\bar{u}_\tau)^T \right)\right\} \right]d\tau \\
& \quad + \sum_{j=1}^d\int_{t_0}^{t_0+\varepsilon} e^{-r\tau} \E \left[\opn{Tr} \left\{  P_{t_0}\left( D_x\sigma^j(\bar{X}_\tau,\bar{u}_\tau) Y^{\varepsilon}_\tau D_x\sigma^j(\bar{X}_\tau,\bar{u}_\tau)^T \right)\right\} \right]d\tau \\
& \quad + \sum_{j=1}^d\int_{t_0}^{t_0+\varepsilon} e^{-r\tau} \E \left[\opn{Tr} \left\{  P_{t_0}\left( \delta \sigma^j(\bar{X}_\tau,\bar{u}_\tau) \left(\delta \sigma^j(\bar{X}_\tau,\bar{u}_\tau) \right)^T \right)\right\} \right]d\tau \\
& \quad + \sum_{j=1}^d\int_{t_0}^{t_0+\varepsilon} e^{-r\tau} \E \left[\opn{Tr} \left\{  P_{t_0}\left( D_x\sigma^j(\bar{X}_\tau,\bar{u}_\tau) y^{\varepsilon}_\tau \left(\delta\sigma^j(\bar{X}_\tau,\bar{u}_\tau)\right)^T \right)\right\} \right]d\tau \\
& \quad + \sum_{j=1}^d\int_{t_0}^{t_0+\varepsilon} e^{-r\tau} \E \left[ \opn{Tr} \left\{  P_{t_0}\left( \delta\sigma^j(\bar{X}_\tau,\bar{u}_\tau) (y^{\varepsilon}_\tau)^T D_x\sigma^j(\bar{X}_\tau,\bar{u}_\tau)^T \right)\right\} \right]d\tau.
\end{split}
\end{equation*}
We will show that using the estimate for $y^\varepsilon_s$ in the form of \eqref{eq:prop:exp_rewritten}, the only term which is not of order  $o(\varepsilon)$ is the third one. Therefore, the final equality will read
\begin{equation}\label{eq:pyy_preliminary}
\begin{split}
&e^{-r\left(t_0+\varepsilon\right)}\E\braket{P_{t_0}y_{t_0+\varepsilon}^\varepsilon,y_{t_0+\varepsilon}^\varepsilon}\\
&= \sum_{j=1}^d \E \int_{t_0}^{t_0+\varepsilon} e^{-r\tau} \Big<P_{t_0} \delta \sigma^j(\bar{X}_\tau,\bar{u}_\tau), \delta \sigma^j(\bar{X}_\tau,\bar{u}_\tau)\Big>d\tau + o(\varepsilon).
\end{split}
\end{equation}
For sake of completeness, let us estimate the second term as
\begin{equation*}
\begin{split}
& \sum_{j=1}^d\int_{t_0}^{t_0+\varepsilon} e^{-r\tau} \E \left[ \opn{Tr} \left\{  P_{t_0}\left( D_x\sigma^j(\bar{X}_\tau,\bar{u}_\tau) Y^{\varepsilon}_\tau D_x\sigma^j(\bar{X}_\tau,\bar{u}_\tau)^T \right)\right\} \right]d\tau \\
&= \sum_{j=1}^d\int_{t_0}^{t_0+\varepsilon} e^{-r\tau} \E \big<P_{t_0}D_x\sigma^j(\bar{X}_\tau,\bar{u}_\tau) y^{\varepsilon}_\tau, D_x\sigma^j(\bar{X}_\tau,\bar{u}_\tau) y^{\varepsilon}_\tau\big>d\tau \\
&\leq \sum_{j=1}^d\int_{t_0}^{t_0+\varepsilon}  e^{-r\tau} \E \left[ \norm{P_{t_0}}_2 \abs{D_x\sigma^j(\bar{X}_\tau,\bar{u}_\tau)}^2\abs{y^{\varepsilon}_\tau}^2 \right]d\tau \\
&\leq \sum_{j=1}^d \int_{t_0}^{t_0+\varepsilon}\E\left(  e^{-r\tau}\norm{P_{t_0}}_2^2 \right)^{1/2}\E\left(  e^{-r\tau}\abs{D_x\sigma^j(\bar{X}_\tau,\bar{u}_\tau)}^4 \right)^{1/2}\E\left(  e^{-r\tau}\abs{y^{\varepsilon}_\tau}^4 \right)^{1/2} d\tau,
\end{split}
\end{equation*}
and the order of $o(\varepsilon)$ now follows by Proposition \ref{p:def:P_t}, the polynomial growth of $D_x\sigma(\cdot)$ and estimate \eqref{eq:prop:exp_rewritten}, respectively. The remaining terms can be treated in the similar way. \smallskip

\noindent To finalize the proof of \eqref{eq:pyy_preliminary}, it remains to be shown that
\begin{equation}
\sum_{j=1}^d \E \int_{t_0}^{t_0+\varepsilon} e^{-r\tau} \Big<\left( P_\tau - P_{t_0}\right) \delta \sigma^j(\bar{X}_\tau,\bar{u}_\tau), \delta \sigma^j(\bar{X}_\tau,\bar{u}_\tau)\Big>d\tau = o(\varepsilon).
\end{equation}
But this is easily obtained by repeating the same arguments as in the proof of \textbf{(i)}. The proof of the Proposition is now concluded.
\end{proof}

\section{Necessary stochastic maximum principle}

For our main result we need to recall the notion of the Hamiltonian of the system. Given the control problem \eqref{SDE}-\eqref{functional}, let us define $H: \R^n \times U \times \R^n \times \R^{n\times d} \to \R$ as
\begin{equation}\label{Hamiltonian}
H(x,u,p,q) = \braket{p,b(x,u)} + \opn{Tr}\left[ q^T \sigma(x,u) \right] - f(x,u).
\end{equation}

\begin{theorem}\label{main_thm}
Assume (H1)-(H5) hold and let $(\bar{X},\bar{u})$ be an optimal pair. Then there exist $r >0$, a pair $(p,q) \in L^{2,-r}_{\mathcal{F}}(\R_+; \R^n) \times L^{2,-r}_{\mathcal{F}}(\R_+; \R^{n\times d})$ and a progressively measurable process $(P_t)_{t \ge 0}$ with values in $\mathcal{S}^n$ such that the following variational inequality holds, $\mP \otimes dt-$a.e.
\begin{equation*}
H(\bar{X}_t,v,p_t,q_t) -  H(\bar{X}_t,\bar{u}_t,p_t,q_t) + \frac{1}{2}\sum^d_{j=1}\Big<P_t\left( \sigma^j(\bar{X}_t,v) - \sigma^j(\bar{X}_t,\bar{u}_t) \right), \sigma^j(\bar{X}_t,v) - \sigma^j(\bar{X}_t,\bar{u}_t)\Big> \le 0,
\end{equation*}
for every $v \in U$. The pair of processes $(p,q)$ is the unique solution to the first adjoint equation \eqref{eq.adjoint_first}. The definition of the process $P_t$ is given in Proposition \ref{p:def:P_t} and the process satisfies $\sup_{t\ge 0} \E\norm{P_t}_2^2 < \infty$.
\end{theorem}

\begin{remark}\label{rm:r}
\normalfont
A sufficient condition for such $r$ is given in the Appendix.
\end{remark}
Before proving the theorem, it is useful to rewrite the variation of cost functional in a suitable form, as the following proposition suggests.
\begin{proposition}\label{prop:difference functionals}
There exists $r >0$ such that the following expansion holds
\begin{equation}\label{eq:spike_cost2}
\begin{split}
J\lp u^{\varepsilon} (\cdot) \rp - J \lp \bar{u}(\cdot) \rp &= \E \int_0^{\infty} e^{-rt} \left[ -\sum_{j=1}^{d} \braket{q^j_t,\delta\sigma^j(\bar{X}_t,\bar{u}_t)} - \braket{p_t,\delta b(\bar{X}_t,\bar{u}_t)} + \delta f(\bar{X}_t,\bar{u}_t) \right]dt \\
&\quad - \frac{1}{2} \E \int_0^{\infty}e^{-rt} \opn{Tr}\left[ D_x^2H(\bar{X}_t,\bar{u}_t,p_t,q_t)y^{\varepsilon}_t\left(y^{\varepsilon}_t\right)^T \right]dt + o(\varepsilon),
\end{split}
\end{equation}
where $H(\bar{X}_t,\bar{u}_t,p_t,q_t)$ is the Hamiltonian of the system computed along the optimal trajectory.
\end{proposition}

\begin{proof}
See Appendix.
\end{proof}

\noindent Now we are in position to end the proof of the SMP.

\begin{proof}[Proof of Theorem \ref{main_thm} ]
The difficult step of the proof is to show that the following holds
\begin{equation}\label{eq:final_relation}
\E\int_0^\infty e^{-rs} \braket{D_x^2 H(s) y^\varepsilon_s,y^\varepsilon_s}ds = \sum_{j=1}^d \E\int_0^\infty e^{-rs} \Big<P_s\delta \sigma^j(\bar{X}_s,\bar{u}_s),\delta \sigma^j(\bar{X}_s,\bar{u}_s)\Big>ds + o(\varepsilon).
\end{equation}
Indeed, if relation \eqref{eq:final_relation} holds true then by using Proposition \ref{prop:difference functionals} we get
\begin{equation}\label{eq:spike_cost3}
\begin{split}
0 &\le J\lp u^{\varepsilon} (\cdot) \rp - J \lp \bar{u}(\cdot) \rp  \\
&= \E \int_0^{\infty} e^{-rs} \left[ -\sum_{j=1}^{d} \braket{q^j_s,\delta\sigma^j(\bar{X}_s,\bar{u}_s)} - \braket{p_s,\delta b(\bar{X}_s,\bar{u}_s)} + \delta f(\bar{X}_s,\bar{u}_s) \right]dt \\
&\quad -\frac{1}{2}\sum_{j=1}^d \E\int_0^\infty e^{-rs} \Big<P_s\delta \sigma^j(\bar{X}_s,\bar{u}_s),\delta \sigma^j(\bar{X}_s,\bar{u}_s)\Big>ds + o(\varepsilon),
\end{split}
\end{equation}
thanks to the optimality of $\bar{u}(\cdot)$. Now the final variational inequality follows by standard arguments, i.e. by using the definition of
$\delta \sigma^j, \delta b, \delta f$, noting that $E_{\varepsilon} = [t_0, t_0 + \varepsilon]$ and by sending $\varepsilon \searrow 0$. \medskip

\noindent Let us focus on the proof of \eqref{eq:final_relation}. Recalling Remark \ref{rem:y}, we can rewrite the left hand side of \eqref{eq:final_relation} in the following form
\begin{equation*}
\begin{split}
\E&\int_0^\infty e^{-rs} \braket{D_x^2 H(s) y^\varepsilon_s,y^\varepsilon_s}ds \\
&= \E\int_{t_0}^{t_0+\varepsilon} e^{-rs} \braket{D_x^2 H(s) y^\varepsilon_s,y^\varepsilon_s}ds + \E\int_{t_0 + \varepsilon}^{\infty} e^{-rs} \braket{D_x^2 H(s) y^\varepsilon_s,y^\varepsilon_s}ds \\
&= \E\int_{t_0 + \varepsilon}^{\infty} e^{-rs} \Big<D_x^2 H(s) y_s^{t_0+\varepsilon, y^\varepsilon_{t_0+\varepsilon}},y_s^{t_0+\varepsilon, y^\varepsilon_{t_0+\varepsilon}}\Big>ds + o(\varepsilon), \\
\end{split}
\end{equation*}
where we have used Proposition \ref{p:expansion} to estimate the first integral and the identity $y_s^\varepsilon = y_s^{t_0+\varepsilon, y^\varepsilon_{t_0+\varepsilon}}$, for $s \ge t_0 +\varepsilon$.
Taking into account Remark \ref{rem:def_P_RV}, we finally deduce the following decomposition
\begin{equation*}
\begin{split}
\E\int_0^\infty e^{-rs} \braket{D_x^2 H(s) y^\varepsilon_s,y^\varepsilon_s}ds &= e^{-r\left(t_0+\varepsilon\right)}\E\braket{P_{t_0+\varepsilon}y_{t_0+\varepsilon}^\varepsilon,y_{t_0+\varepsilon}^\varepsilon} + o(\varepsilon)\\
&= e^{-r\left(t_0+\varepsilon\right)}\E\Big<\left( P_{t_0+\varepsilon} - P_{t_0}\right)y_{t_0+\varepsilon}^\varepsilon,y_{t_0+\varepsilon}^\varepsilon\Big> \\
& \quad + e^{-r\left(t_0+\varepsilon\right)}\E\big<P_{t_0}y_{t_0+\varepsilon}^\varepsilon,y_{t_0+\varepsilon}^\varepsilon\big> + o(\varepsilon).
\end{split}
\end{equation*}
Finally, by Proposition \ref{p:last:estimtes}, the proof of the Theorem is now concluded.
\end{proof}

\section{acknowledgements}
The authors wish to thank to Marco Fuhrman and Gianmario Tessitore for encouragement and for many valuable discussions.

\section{appendix}

\subsection{Proof of Proposition \ref{p:expansion}}

\begin{proof}
In the following we are going to linearize the equations satisfied by $\xi^{\varepsilon}(\cdot), \eta^{\varepsilon}(\cdot)$ and $\zeta^{\varepsilon}(\cdot)$ in order to use the estimate obtained in Lemma \ref{l:4.1}.\\
\newline
\textbf{(i)} It is easy to see that the equation for $\xi^{\varepsilon}(\cdot)$ can be rewritten in the form
\begin{equation*}
d\xi^{\varepsilon}_t= \bigl[ G_b(t)\xi^{\varepsilon}_t + \delta b_t\chi_{E_{\varepsilon}}(t)\bigr]dt + \sum_{j=1}^d \bigl[ G^j_{\sigma}(t)\xi^{\varepsilon}_t + \delta \sigma^j_t\chi_{E_{\varepsilon}}(t)\bigr]dW^j_t,
\end{equation*}
where
\[ G_b(t):= \int_0^1 D_xb(\bar{X}_t + \theta\xi^{\varepsilon}_t,u^{\varepsilon}_t)d\theta, \qquad  G_{\sigma}^j(t):= \int_0^1 D_x\sigma^j(\bar{X}_t + \theta\xi^{\varepsilon}_t, u^{\varepsilon}_t)d\theta. \]
Thanks to Hypothesis (H4) we can apply Lemma \ref{l:4.1} and obtain (the constant $K>0$ varies from line to line)

\begin{equation}\label{eq:deltab}
\begin{split}
\sup_{t\in \R_+} e^{-rkt} \E\abs{\xi^{\varepsilon}_t}^{2k} &\leq K\Bigl[ \int_0^{\infty} e^{-\frac{r}{2}t}\left( \E\abs{\delta b_t\chi_{E_{\varepsilon}}(t)}^{2k}\right)^{\frac{1}{2k}} dt\Bigr]^{2k} \\
& \quad + K\sum_{j=1}^d \Bigl[ \int_0^{\infty} e^{-rt} \left( \E\abs{\delta \sigma^{j}_t\chi_{E_{\varepsilon}}(t)}^{2k} \right)^{\frac{1}{k}} dt\Bigr]^{k}\\
&\leq K \Bigl[ \int_{E_{\varepsilon}} e^{-\frac{r}{2}t}\left( \E\abs{b(\bar{X}_t,u^{\varepsilon}_t) - b(\bar{X}_t,\bar{u}_t)}^{2k}\right)^{\frac{1}{2k}}dt\Bigr]^{2k} \\
& \quad + K\sum_{j=1}^d \Bigl[ \int_{E_{\varepsilon}} e^{-rt} \left( \E\abs{\sigma^j_t(\bar{X}_t,u^{\varepsilon}_t) - \sigma^{j}_t(\bar{X}_t,\bar{u}_t)}^{2k} \right)^{\frac{1}{k}} dt\Bigr]^{k} \\
&\leq K[\varepsilon^{2k} + \varepsilon^{k}] \leq K\varepsilon^{k},
\end{split}
\end{equation}
thanks to the polynomial growth of the coefficients and the boundedness of the integration interval $E_{\varepsilon}$. Indeed, remember that it is easy to control all the moments of $\bar{X}$ up to a fixed time. In this case the discount factor $\rho_1$ can be chosen equal to the initial one $\rho_1 = r$.\\
\newline
\textbf{(ii)}
Using again the global monotonicity assumption and Lemma \ref{l:4.1}, the estimate for $y^{\varepsilon}$ follows in the same way.\\
\newline
\textbf{(iii)} For $z^{\varepsilon}$ we start by estimating its norm in the space $L_{\mathcal{F}}^{2k, -rk\alpha}(\R_+;\R^n)$, for a generic $\alpha\in \R$. Using the same technique as in Lemma \ref{l:4.1}, we obtain
\begin{equation*}
\begin{split}
\sup_{t\in \R_+}e^{-rk\alpha t}\E\abs{z^{\varepsilon}_t}^{2k} &\leq K\Bigl[ \int_0^{\infty} e^{-\frac{r\alpha}{2} t}\left( \E \abs{\delta b_t\chi_{E_{\varepsilon}}(t) + \dfrac{1}{2}D^2_xb(\bar{X}_t,\bar{u}_t)(y^{\varepsilon}_t)^2}^{2k} \right)^{\frac{1}{2k}}dt\Bigr]^{2k} \\
& \quad + K \sum_{j=1}^d \Bigl[ \int_0^{\infty} e^{-r\alpha t}\Bigl(\E\abs{\delta (D_x\sigma^j_t) \chi_{E_{\varepsilon}}(t)y^{\varepsilon}_t + \dfrac{1}{2}D^2_x\sigma^j(\bar{X}_t,\bar{u}_t) (y^{\varepsilon}_t)^2}^{2k}\Bigr)^{\frac{1}{k}}dt\Bigr]^{k}.\\
\end{split}
\end{equation*}

\noindent The first term (with $\delta b_t$) can be treated as before, thanks to the boundedness of $E_\varepsilon$. Let us discuss the second one. It holds
\begin{equation*}
\begin{split}
&\int_0^{\infty} e^{-\frac{r\alpha}{2}t}\left( \E \abs{D^2_xb(\bar{X}_t,\bar{u}_t)(y^{\varepsilon}_t)^2}^{2k} \right)^{\frac{1}{2k}}dt\\
 &\leq \int_0^{\infty} \left( e^{-r\alpha kt} \E \abs{D^2_xb(\bar{X}_t,\bar{u}_t)}^{4k} \right)^{\frac{1}{4k}} \left( e^{-r\alpha kt} \E \abs{y^{\varepsilon}_t}^{8k} \right)^{\frac{1}{4k}} dt \\
&\leq \left( \sup_{t \ge 0} e^{-r\alpha kt} \E \abs{y^{\varepsilon}_t}^{8k} \right)^{\frac{1}{4k}} \int_0^{\infty} e^{-\frac{r\alpha}{4} t}\left( \E \abs{D^2_xb(\bar{X}_t,\bar{u}_t)}^{4k} \right)^{\frac{1}{4k}}dt \\
&\leq K\left( \sup_{t \ge 0} e^{-r\alpha kt} \E \abs{y^{\varepsilon}_t}^{8k} \right)^{\frac{1}{4k}} \int_0^{\infty} e^{-\frac{r\alpha}{4} t}\left( \E \abs{1 + \abs{\bar{X}_t}^{2m+1}}^{4k} \right)^{\frac{1}{4k}}dt \\
&\leq K\left( \sup_{t \ge 0} e^{-r\alpha kt} \E \abs{y^{\varepsilon}_t}^{8k} \right)^{\frac{1}{4k}} \left[ \int_0^{\infty} e^{-\frac{r\alpha}{4} t}dt + \int_0^{\infty} e^{-\frac{r\alpha}{4} t} \left( \E \abs{\bar{X}_t}^{4k(2m+1)} \right)^{\frac{1}{4k}}dt \right] \\
&\leq K\left( \sup_{t \ge 0} e^{-r\alpha kt} \E \abs{y^{\varepsilon}_t}^{8k} \right)^{\frac{1}{4k}} \left[ \int_0^{\infty} e^{-\frac{r\alpha}{4} t}dt + \left( \int_0^{\infty} e^{-\frac{r\alpha}{4} t} \E \abs{\bar{X}_t}^{4k(2m+1)} dt\right)^{\frac{1}{4k}} \right], \\
\end{split}
\end{equation*}
where we used the polynomial growth of $D^2_x b$ and Jensen inequality, assuming that $\int_0^{\infty} e^{-\frac{r\alpha}{4} t}dt < \infty$, hence $r\alpha > 0$. Moreover, if we choose $\alpha \geq \max \left( 4, 8k(2m+1)\right) = 8k(2m+1)$, we have that
\begin{equation}
\int_0^{\infty} e^{-\frac{r\alpha}{2}t}\left( \E \abs{D^2_xb(\bar{X}_t,\bar{u}_t)(y^{\varepsilon}_t)^2}^{2k} \right)^{\frac{1}{2k}}dt \leq K \varepsilon,
\end{equation}
where $K = K(r,k,m)$ depends on the choice of the initial discount factor, the order of integration and the polynomial growth of the coefficients of the state. Let us briefly sketch also the computations for the last addendum
\begin{equation*}
\begin{split}
&\int_0^{\infty} e^{-r\alpha t}\left( \E \abs{D^2_x\sigma^j(\bar{X}_t,\bar{u}_t)(y^{\varepsilon}_t)^2}^{2k} \right)^{\frac{1}{k}}dt\\
&\leq K\left( \sup_{t \ge 0} e^{-r\alpha kt} \E \abs{y^{\varepsilon}_t}^{8k} \right)^{\frac{1}{2k}} \int_0^{\infty} e^{-\frac{r\alpha}{2} t}\left( \E \abs{D^2_x\sigma^j(\bar{X}_t,\bar{u}_t)}^{4k} \right)^{\frac{1}{2k}}dt \\
&\leq K\left( \sup_{t \ge 0} e^{-r\alpha kt} \E \abs{y^{\varepsilon}_t}^{8k} \right)^{\frac{1}{2k}} \int_0^{\infty} e^{-\frac{r\alpha}{2} t}\left( \E \abs{1 + \abs{\bar{X}_t}^{m}}^{4k} \right)^{\frac{1}{2k}}dt. \\
\end{split}
\end{equation*}
If $\alpha \ge 4km$,  following the same strategy as above we end up with
\begin{equation}
\int_0^{\infty} e^{-r\alpha t}\left( \E \abs{D^2_x\sigma^j(\bar{X}_t,\bar{u}_t)(y^{\varepsilon}_t)^2}^{2k} \right)^{\frac{1}{k}}dt \leq K\varepsilon^2.
\end{equation}
Summing up all the estimates and using Lemma \ref{l:4.1}, we easily get the required result, for some $\rho_3$ big enough. In this case it is sufficient to choose $\rho_3 \ge \alpha r\ge 8k(2m+1)r$.\smallskip

\noindent \textbf{(iv)} Following Yong and Zhou \citep{Yo99}, it is easy to see that
\begin{equation*}
d\eta^{\varepsilon}_t = \left[ D_xb(\bar{X}_t,\bar{u}_t)\eta^{\varepsilon}_t + A^{\varepsilon}_t \right]dt + \sum_{j=1}^d \left[ D_x\sigma^j(\bar{X}_t,\bar{u}_t) \eta^{\varepsilon}_t + B^{j,\varepsilon}_t\right] dW^j_t,
\end{equation*}
where
\begin{equation*}
\begin{split}
A^{\varepsilon}_t &:= \delta b_t\chi_{E_{\varepsilon}}(t) + \left[G_b(t) - D_xb(\bar{X}_t,\bar{u}_t)\right]\xi^{\varepsilon}_t;\\
B^{j,\varepsilon}_t &:= \bigl(G_{\sigma}^j(t) - D_x^j\sigma(\bar{X}_t,\bar{u}_t)\bigr)\xi^{\varepsilon}_t.
\end{split}
\end{equation*}
Let us consider $A^{\varepsilon}(\cdot)$ first.
\begin{equation*}
\begin{split}
\int_0^{\infty} &e^{-\frac{r\alpha}{2}t} \left( \E \abs{A^{\varepsilon}_t}^{2k} \right)^{\frac{1}{2k}} dt \\
&\leq K \int_0^{\infty} e^{-\frac{r\alpha}{2}t} \Bigl( \E \abs{\delta b_t\chi_{E_{\varepsilon}}(t)}^{2k} \Bigr)^{\frac{1}{2k}}dt\\
& \quad + K \int_0^{\infty} e^{-\frac{r\alpha}{2}t} \Bigl( \E\big|\bigl[ G_b(t) - D_xb(\bar{X}_t,\bar{u}_t)) \bigr]\xi^{\varepsilon}_t \big|^{2k} \Bigr)^{\frac{1}{2k}} dt\\
&\leq K \varepsilon + K\left( \sup_{t \ge 0}e^{-r\alpha kt}\E\abs{\xi^{\varepsilon}_t}^{4k} \right)^{\frac{1}{4k}} \int_0^{\infty} e^{-\frac{r\alpha}{4}t} \left( \E \abs{G_b(t) - D_xb(\bar{X}_t,\bar{u}_t))}^{4k}\right)^{\frac{1}{4k}}dt \\
&\leq K \varepsilon + K\varepsilon^{1/2} \int_0^{\infty} e^{-\frac{r\alpha}{4}t} \left( \E \abs{G_b(t) - D_xb(\bar{X}_t,\bar{u}_t))}^{4k}\right)^{\frac{1}{4k}}dt, \\
\end{split}
\end{equation*}
due to the previous result with $\alpha \geq 2$, and H\"{older} inequality. Regarding the last term we have
\begin{equation*}
\begin{split}
G_b(t) - D_xb(\bar{X}_t,\bar{u}_t)) &= \int_0^1\bigl[ D_xb(\bar{X}_t + \theta\xi^{\varepsilon}_t,u^{\varepsilon}_t) - D_xb(\bar{X}_t,\bar{u}_t))\bigr] d\theta \\
&= \int_0^1\bigl[D_xb(\bar{X}_t + \theta\xi^{\varepsilon}_t,u^{\varepsilon}_t) - D_xb(\bar{X}_t + \theta \xi^{\varepsilon}_t,\bar{u}_t)\bigr] d\theta \\
& \quad + \int_0^1\bigl[ D_xb(\bar{X}_t + \theta\xi^{\varepsilon}_t,\bar{u}_t) - D_xb(\bar{X}_t,\bar{u}_t))\bigr] d\theta.
\end{split}
\end{equation*}
Hence, using the Taylor expansion with Lagrange rest, there exists $\tilde{x}$ (depending on $t$ and $\omega$) such that

\begin{equation*}
\begin{split}
\int_0^{\infty} &e^{-\frac{r\alpha}{4}t} \abs{G_b(t) - D_xb(\bar{X}_t,\bar{u}_t)}_{L^{4k}(\Omega)}dt\\
 &\leq \int_{E_{\varepsilon}} e^{-\frac{r\alpha}{4}t}\left( \E \Big| \int_0^1 D_xb(\bar{X}_t + \theta\xi^{\varepsilon}_t,u^{\varepsilon}_t) - D_xb(\bar{X}_t + \theta \xi^{\varepsilon}_t,\bar{u}_t)\bigr] d\theta \Big|^{4k} \right)^{\frac{1}{4k}} dt\\
& \quad +\int_0^{\infty} e^{-\frac{r\alpha}{4}t}\left( \E \Big|\int_0^1 D_x^2b(\tilde{x},\bar{u}_t)\theta\xi^{\varepsilon}(t) d\theta\Big|^{4k}\right)^{\frac{1}{4k}} dt \\
&\leq K\varepsilon + \left( \sup_{t\ge 0}e^{-r\alpha kt}\E\abs{\xi^{\varepsilon}_t}^{8k}\right)^{\frac{1}{8k}}\int_0^{\infty} e^{-\frac{r\alpha}{8}t} \Bigl(\E\abs{D_x^2b(\tilde{x},\bar{u}_t)}^{8k}\Bigr)^{\frac{1}{8k}}dt \\
&\leq  K(\varepsilon + \varepsilon^{1/2}),
\end{split}
\end{equation*}
thanks to the estimate obtained in  point $\textbf{(i)}$ and the polynomial growth of $D_x^2b$ (here we have to require $\alpha \ge 32k(2m+1)$). Then
\begin{equation*}
\begin{split}
\int_0^{\infty} e^{-\frac{r\alpha }{2}t}\left( \E \abs{A^{\varepsilon}_t}^{2k} \right)^{\frac{1}{2k}} dt &\leq K \varepsilon + K\varepsilon^{1/2} \int_0^{\infty} e^{-\frac{r\alpha}{4}t} \left( \E \abs{G_b(t) - D_xb(\bar{X}_t,\bar{u}_t)}^{4k} \right)^{\frac{1}{4k}}dt \\
&\leq K \varepsilon.
\end{split}
\end{equation*}
For $B^{j,\varepsilon}_t$, proceeding in a similar way we obtain
\[ \int_0^{\infty} e^{- r\alpha t} \Bigl( \E\abs{G_{\sigma}^j(t) - D_x\sigma^j(\bar{X}_t,\bar{u}_t)}^{4k} \Bigr)^{\frac{1}{2k}}dt \leq K\varepsilon. \]
To conclude, we apply Lemma \ref{l:4.1} to get
\begin{equation*}
\begin{split}
\sup_{t\in \R_+} e^{-r\alpha kt}\E\abs{\eta^{\varepsilon}_t}^{2k} &\leq K\left( \int_0^{\infty} e^{-\frac{r\alpha}{2}t} \left( \E \abs{A^{\varepsilon}_t}^{2k} \right)^{\frac{1}{2k}}dt\right)^{2k} + K \sum_{j=1}^d \left( \int_0^{\infty} e^{-r\alpha t} \bigl(\E\abs{B^{j,\varepsilon}_t}^{2k}\bigr)^{\frac{1}{k}}dt \right)^k \\
&\leq K \bigl(\varepsilon^{2k} + \varepsilon^{2k}\bigr) = O(\varepsilon^{2k}).
\end{split}
\end{equation*}

In this case $\rho_4$ can be chosen as $\rho_4 \ge r\alpha \ge 32k(2m+1)r$.\smallskip

\noindent \textbf{(v)} Let us denote $d\zeta^{\varepsilon}(t) = d(\eta^{\varepsilon}(t) - \xi^{\varepsilon}(t))$,
\[\begin{cases}
d\zeta^{\varepsilon}_t = \left( D_xb(\bar{X}_t,u_t)\zeta^{\varepsilon}_t + A^{\varepsilon}_t \right)dt + \sum_{j=1}^d\left( D_x\sigma^j(\bar{X}_t,u_t)\zeta^{\varepsilon}_t + B^{j,\varepsilon}_t \right)dW^j_t,\\
\zeta^{\varepsilon}(0) = 0,
\end{cases}\]
where
\begin{equation*}
\begin{split}
A^{\varepsilon}_T &:= \delta D_xb_t\chi_{E_{\varepsilon}}(t)\xi^{\varepsilon}_t + \dfrac{1}{2}\bigl[ \widetilde{G}_{b}(t) - D_x^2b(\bar{X}_t,u^{\varepsilon}_t) \bigr](\xi^{\varepsilon}_t)^2 \\
& \quad + \dfrac{1}{2}\delta D_x^2b_t\chi_{E_{\varepsilon}}(t)(\xi^{\varepsilon}_t)^2 + \dfrac{1}{2}D_x^2b(\bar{X}_t,u_t)[ (\xi^{\varepsilon}_t)^2 - (y^{\varepsilon}_t)^2 ],\\
&\quad \\
B^{\varepsilon}_t &:= \delta D_x\sigma_t\chi_{E_{\varepsilon}}(t)\eta^{\varepsilon}_t + \dfrac{1}{2}\bigl[ \widetilde{G}_{\sigma}(t) - D_x^2\sigma(\bar{X}_t,u^{\varepsilon}_t) \bigr](\xi^{\varepsilon}_t)^2 \\
& \quad + \dfrac{1}{2}\delta D_x^2\sigma_t\chi_{E_{\varepsilon}}(t)(\xi^{\varepsilon}_t)^2 + \dfrac{1}{2}D_x^2\sigma(\bar{X}_t,u_t) [(\xi^{\varepsilon}_t)^2 - (y^{\varepsilon}_t)^2 ],
\end{split}
\end{equation*}
and
\[\begin{cases}
\widetilde{G}_{b}(t) := 2 \int_0^1 \theta D_x^2b(\theta\bar{X}_t + (1-\theta)X^{\varepsilon}_t,u^{\varepsilon}_t)d\theta,\\
\widetilde{G}_{\sigma}(t) := 2 \int_0^1 \theta D_x^2\sigma(\theta\bar{X}_t + (1-\theta)X^{\varepsilon}_t,u^{\varepsilon}_t)d\theta.
\end{cases}\]

First, let us consider the $A^{\varepsilon}(\cdot)$ term. Applying the H\"older inequality gives
\begin{equation*}
\begin{split}
\int_0^\infty &e^{-\frac{r\alpha}{2}t}\left( \E \abs{A^{\varepsilon}_t}^{2k} \right)^{\frac{1}{2k}}dt \\
&\leq \int_0^\infty e^{-\frac{r}{2}t}\Bigl[ \bigl( \E \abs{\delta D_xb_t\chi_{E_{\varepsilon}}(t)\xi^{\varepsilon}_t}^{2k} \bigr)^{\frac{1}{2k}} + \dfrac{1}{2}\bigl( \E\abs{\bigl[ \widetilde{G}_b(t) - D_x^2b(\bar{X}(t),u^{\varepsilon}_t) \bigr](\xi^{\varepsilon}_t)^2}^{2k} \bigr)^{\frac{1}{2k}} \\
& \quad + \dfrac{1}{2}\bigl( \E \abs{\delta D_x^2b_t\chi_{E_{\varepsilon}}(t)(\xi^{\varepsilon}_t)^2}^{2k}\bigr)^{\frac{1}{2k}} + \dfrac{1}{2}\bigl( \E\abs{ D_x^2b(\bar{X}_t,u_t)[ (\xi^{\varepsilon}_t)^2 - (y^{\varepsilon}_t)^2]}^{2k} \bigr)^{\frac{1}{2k}}\Bigr]dt \\
&\leq K\bigl( \sup_{t \ge 0} e^{-r\alpha kt}\E\abs{\xi^{\varepsilon}_t}^{4k} \bigr)^{\frac{1}{4k}} \int_{E_{\varepsilon}}e^{-\frac{r\alpha}{4}t}\bigl( \E\abs{\delta D_xb_t}^{4k} \bigr)^{\frac{1}{4k}}dt \\
& \quad + K\bigl( \sup_{t \ge 0}e^{-r\alpha kt} \E\abs{\xi^{\varepsilon}_t}^{8k} \bigr)^{\frac{1}{4k}}\int_0^\infty e^{-\frac{r\alpha}{4}t} \bigl( \E\abs{\widetilde{G}_b(t) - D_x^2b(\bar{X}_t,u^{\varepsilon}_t)}^{4k} \bigr)^{\frac{1}{4k}}dt\\
& \quad +K\bigl( \sup_{t \ge 0} e^{-r\alpha kt}\E\abs{\xi^{\varepsilon}_t}^{8k} \bigr)^{\frac{1}{4k}}\int_{E_{\varepsilon}} e^{-\frac{r\alpha}{4}t} \bigl( \E\abs{\delta D_x^2b_t}^{4k} \bigr)^{\frac{1}{4k}}dt \\
& \quad + K \bigl(\sup_{t \ge 0} e^{-r\alpha kt} \E\abs{\eta^{\varepsilon}_t}^{8k} \bigr)^{\frac{1}{8k}}\bigl(\sup_{t \ge 0} e^{-r\alpha kt} \E\abs{\xi^{\varepsilon}_t + y^{\varepsilon}_t}^{8k} \bigr)^{\frac{1}{8k}}\int_0^\infty e^{-\frac{r\alpha t}{4}} \bigl( \E\abs{D_x^2b(\bar{X}_t,u_t)}^{4k} \bigr)^{\frac{1}{4k}}dt. \\
\end{split}
\end{equation*}
If $\alpha \ge 4$ the first and the third term con be easily controlled. For the last addendum we use the same technique as in $\textbf{(iii)}$ to get the boundedness of the integral for $\alpha \geq 8k(2m+1)$, hence

\begin{equation*}
\int_0^\infty e^{-\frac{r\alpha}{2}t}\left( \E \abs{A^{\varepsilon}_t}^{2k} \right)^{\frac{1}{2k}}dt \leq K\Big[ \varepsilon^{3/2} + \varepsilon \int_0^\infty e^{-\frac{r\alpha}{4}t} \bigl( \E\abs{\widetilde{G}_b(t) - D_x^2b(\bar{X}_t,u^{\varepsilon}_t)}^{4k} \bigr)^{\frac{1}{4k}}dt  + \varepsilon^2 + \varepsilon^{3/2}\Big].
\end{equation*}

Finally, we can rewrite $\E\abs{\widetilde{G}_b(t) - D_x^2b(t,\bar{x}(t),u^{\varepsilon}(t))}^{4k}$ in the following form
\begin{equation}
\begin{split}
&\widetilde{G}_{b}(t) - D_x^2b(\bar{X}_t,u^{\varepsilon}_t)= \\
&=2 \int_0^1 \theta D_x^2b(\theta\bar{X}_t + (1-\theta)X^{\varepsilon}_t,u^{\varepsilon}_t)d\theta - D_x^2b(\bar{X}_t,u^{\varepsilon}_t)\\
&= 2 \int_0^1 \theta \Bigl[ D_x^2b(\theta\bar{X}_t + (1-\theta)X^{\varepsilon}_t,u^{\varepsilon}_t) - D_x^2b(\bar{X}_t,u^{\varepsilon}_t)\Bigr] d\theta.
\end{split}
\end{equation}
If $\alpha \geq 8k(2m+1)$, by the continuity of the map $x \mapsto D_x^2b(x,u)$ and dominated convergence theorem it follows that
\begin{equation*}
\int_0^\infty e^{-\frac{r\alpha}{4}t} \left( \int_0^1 \theta \Bigl[ D_x^2b(\theta\bar{X}_t + (1-\theta)X^{\varepsilon}_t,u^{\varepsilon}_t) - D_x^2b(\bar{X}_t,u^{\varepsilon}_t)\Bigr] d\theta\right)^{\frac{1}{4k}}dt \to 0,
\end{equation*}
as $\varepsilon \to 0$. Then
\begin{equation*}
\begin{split}
\int_0^\infty e^{-\frac{r\alpha}{2}t}\left( \E \abs{A^{\varepsilon}_t}^{2k} \right)^{\frac{1}{2k}}dt \leq K[ \varepsilon^{3/2} + \varepsilon^2 + \varepsilon^{3/2} ] + o(\varepsilon) = o(\varepsilon).
\end{split}
\end{equation*}

\noindent For $B^{\varepsilon}(t)$ we proceed in the same way to obtain
\begin{equation*}
\int_0^\infty e^{-\frac{rt}{2k}} \left( \E \abs{B^{j,\varepsilon}(t)}^{2k} \right)^{\frac{1}{k}}dt = o(\varepsilon^2).
\end{equation*}
Using Lemma \ref{l:4.1}, the desired result follows for $\rho_5 \ge r\alpha \ge 8k(2m+1)r$.
\end{proof}

\subsection{Proof of Proposition \ref{p:estimate_Y}}

\begin{proof}
Since $Y^{\varepsilon}_t = y^{\varepsilon}_t(y^{\varepsilon}_t)^T$, the existence and uniqueness of a solution follow from the existence and uniqueness of the process $y^{\varepsilon}$ (see Theorem \ref{th:Variation equations}), with the restriction $r > 2c_{1/2}$. \smallskip

\noindent Let us now denote $A_t: = D_xb(\bar{X}_t,\bar{u}_t)$, $B^j_t:= D_x\sigma^j(\bar{X}_t,\bar{u}_t)$ and note that a symmetric (positive) definite matrix $Y^{\varepsilon}_t$ can be decomposed as
$Y^{\varepsilon}_t = \sum^n_{i=1} \gamma_i c_i c_i^T,$ where $\gamma_i \geq 0$ for all $i$ and $(c_i)_i$ is an orthonormal basis of $\R^n$.
Clearly, each $\gamma_i$ and $c_i$ depend both on $t$ and $\varepsilon$ but we omit this notation in the proof. \smallskip

\noindent Having in mind the above, one arrives at
\begin{equation*}
\begin{split}
\braket{A_tY^{\varepsilon}_t,Y^{\varepsilon}_t}_2 &= \opn{Tr}\lbrace	 A_tY^{\varepsilon}_t \lp Y^{\varepsilon}_t\rp^T \rbrace = \sum_{i = 1}^n \gamma_i^2 \opn{Tr} \lbrace A_tc_ic_i^Tc_ic_i^T \rbrace  = \sum_{i = 1}^n \gamma_i^2 \opn{Tr} \lbrace c_ic_i^TA_tc_ic_i^T \rbrace \\
&= \sum_{i = 1}^n \gamma_i^2 \braket{A_tc_i,c_i}\opn{Tr} \lbrace c_ic_i^T \rbrace  = \sum_{i = 1}^n \gamma_i^2 \braket{A_tc_i,c_i},
\end{split}
\end{equation*}
\begin{equation*}
\begin{split}
\braket{B^j_tY^{\varepsilon}_t(B^j_t)^T,Y^{\varepsilon}_t}_2 &= \opn{Tr}\Big\{B^j_t Y^{\varepsilon}_t (B^j_t)^T \lp Y^{\varepsilon}_t \rp^T \Big\} = \sum_{i=1}^n \gamma^2_i  \opn{Tr}\Big\{B^j_t c_i (c_i^T (B^j_t)^Tc_i )c_i^T  \Big\} \\
&= \sum_{i=1}^n \gamma^2_i  \braket{(B^j_t)^T c_i, c_i}^2
\leq \sum_{i=1}^n \abs{(B^j_t)^T (\gamma_i c_i)}^2 |c_i|^2 \\
&= \sum_{i=1}^n \abs{(B^j_t)^T (\gamma_i c_i)}^2 = \sum_{i=1}^n \gamma_i^2\abs{(B^j_t)^T c_i}^2,
\end{split}
\end{equation*}
\noindent and
\begin{equation*}
\begin{split}
\norm{B^j_tY^{\varepsilon}_t}^2_2 &= \opn{Tr} \Big\{ B^j_tY^{\varepsilon}_t \lp Y^{\varepsilon}_t\rp^T(B^j_t)^T \Big\}
 = \sum_{i=1}^n \gamma_i^2 \opn{Tr} \Big\{ B^j_tc_ic_i^Tc_ic_i^T (B^j_t)^T \Big\} \\
&= \sum_{i=1}^n \gamma_i^2 \opn{Tr} \Big\{ c_ic_i^T (B^j_t)^TB^j_tc_ic_i^T \Big\} = \sum_{i=1}^n \gamma_i^2 \braket{(B^j_t)^T B^j_tc_i,c_i}\\
&= \sum_{i=1}^n \gamma_i^2 \abs{B^j_tc_i}^2,
\end{split}
\end{equation*}
where we have used the basic properties of the Trace. Using these estimates we are able to prove the following dissipativity condition
\begin{equation*}
\begin{split}
&\braket{A_tY^{\varepsilon}_t +Y^{\varepsilon}_tA_t^T,Y^{\varepsilon}_t}
+ \sum_{j=1}^d\braket{ B_t^jY^{\varepsilon}_t(B_t^j)^T,Y^{\varepsilon}_t} + \sum_{j=1}^d\norm{B_t^jY^{\varepsilon}_t + Y^{\varepsilon}_tB^j_t}^2_2 \\
&\quad \leq 2 \sum_{i = 1}^n \gamma_i^2 \braket{A_tc_i,c_i} + \sum_{j = 1}^d\sum_{i=1}^n \gamma_i^2\abs{(B^j_t)^T c_i}^2
+ 2\sum_{j = 1}^d\sum_{i=1}^n \gamma_i^2 \abs{B^j_tc_i}^2 \\
& \quad = 2\sum_{i = 1}^n \gamma_i^2 \left( \braket{A_tc_i,c_i} + \frac{3}{2}\sum_{j=1}^d\abs{B^j_tc_i}^2 \right) \\
&\quad \leq 2c_{3/2} \sum_{i = 1}^n \gamma_i^2 \abs{c_i}^2 = 2c_{3/2} \norm{Y^{\varepsilon}_t}_2^2.
\end{split}
\end{equation*}
Now, applying the It\^{o} formula to $e^{-rt}\norm{Y^{\varepsilon}_t}_2^2$ we obtain, for every $T >0$ and $\delta>0$
\begin{equation*}
\begin{split}
\E &\big[e^{-rt}\norm{Y^{\varepsilon}_t}_2^2 \big]+ r \E \int_0^{T}e^{-rt}\norm{Y^{\varepsilon}_t}_2^2 dt\\
&= 2\E\int_0^{T}e^{-rt}\left[ \big< Y^{\varepsilon}_t, A_tY^{\varepsilon}_t +Y^{\varepsilon}_tA_t^T \big>
+ \sum_{j=1}^d \big<Y^{\varepsilon}_t, B_t^jY^{\varepsilon}_t(B_t^j)^T \big> \right]dt\\
& \quad + 2\E\int_0^{T}e^{-rt}\braket{Y^{\varepsilon}_t,\Gamma_t} dt
+ \sum_{j=1}^d \E\int_0^{T}e^{-rt} \norm{B_t^jY^{\varepsilon}_t + Y^{\varepsilon}_tB^j_t + \Lambda_t^j}_2^2 dt \\
&\leq  2c_{3/2}\E\int_0^{T}e^{-rt} \norm{Y^{\varepsilon}_t}_2^2 dt+ \delta \E\int_0^{T}e^{-rt} \norm{Y^{\varepsilon}_t}_2^2dt\\
& \quad + \frac{1}{\delta}\E\int_0^{T}e^{-rt} \norm{\Gamma_t}_2^2dt + \sum_{j=1}^d\int_0^{T}e^{-rt} \norm{\Lambda^j_t}_2^2 dt.
\end{split}
\end{equation*}
Hence
\begin{equation*}
(r-2c_{3/2} - \delta)\E\int_0^{T}e^{-rt} \norm{Y^{\varepsilon}_t}_2^2 dt \leq \frac{1}{\delta}\E\int_0^{T}e^{-rt} \norm{\Gamma_t}_2^2dt + \sum_{j=1}^d\int_0^{T}e^{-rt} \norm{\Lambda^j_t}_2^2 dt,
\end{equation*}
and the estimate follows for $r > 2 c_{3/2}$ by sending $T \rightarrow +\infty$.
The final estimate holds for $r > 2 \max  \lbrace c_{1/2}, c_{3/2} \rbrace$.
\end{proof}

\subsection{Proof of Proposition \ref{prop:difference functionals}}

\begin{proof}
If we compute the It\^{o} differential of the processes $e^{-rt}\braket{y^{\varepsilon}_t,p_t}$, where $p_t$ is the solution to the finite horizon equation \eqref{eq:first.adjoint.finite.horizon}, we obtain
\begin{equation*}
\begin{split}
d \Big( e^{-rt}\braket{y^{\varepsilon}_t,p_t} \Big)&= \braket{d(e^{-rt} y^{\varepsilon}_t),p_t} + e^{-rt}\braket{y_t,dp_t} + \sum_{j=1}^d \braket{q^j_t, D_x\sigma^j(\bar{X}_t,\bar{u}_t)y^{\varepsilon}_t + \delta \sigma^j(\bar{X}_t,\bar{u}_t)}dt \\
&= \Big[ -re^{-rt}\braket{y^{\varepsilon}_t,p_t} + e^{-rt}\braket{D_xb(\bar{X}_t,\bar{u}_t) y_t^{\varepsilon},p_t} - e^{-rt}\braket{y^{\varepsilon}_t, D_xb(\bar{X}_t,\bar{u}_t)^T y_t^{\varepsilon}} \Big.\\
& \quad - \sum_{j=1}^d e^{-rt}\braket{y^{\varepsilon}_t, D_x\sigma^j(\bar{X}_t,\bar{u}_t)^T q^j_t} + e^{-rt}\braket{y^{\varepsilon}_t,D_xf(\bar{X}_t,\bar{u}_t)} + re^{-rt}\braket{y^{\varepsilon}_t,p_t}\\
& \Big. \quad + \sum_{j=1}^d \braket{q^j_t, D_x\sigma^j(\bar{X}_t,\bar{u}_t)y^{\varepsilon}_t + \delta \sigma^j(\bar{X}_t,\bar{u}_t)} \Big]dt + M_tdW_t,
\end{split}
\end{equation*}
where the stochastic term is a local martingale with zero mean value (which can be proved by standard localization argument). Hence, by taking expectation we have for all $T >0$
\begin{equation*}
e^{-rT}\E\braket{y^{\varepsilon}_T,p_T} - \E\int_0^T e^{-rt} \braket{y^{\varepsilon}_t, D_xf(\bar{X}_t,\bar{u}_t) }dt  = \sum_{j=1}^d\E\int_0^T e^{-rt} \braket{q^j_t, \delta \sigma^j(\bar{X}_t,\bar{u}_t)}dt,
\end{equation*}
thanks to the fact that $y^{\varepsilon}_0= 0$. Since $(p_t)_{t \geq 0} \in L_{\mathcal{F}}^{2,-r}(\R_+;\R^n)$ then there exists a sequence of times $(T_n)_{n\geq 1}$
with $T_n \nearrow +\infty$ as $n \rightarrow +\infty$ such that along this sequence $\E \big[ e^{-rT_n}p_{T_{n}} \big] \to 0$. Hence, for all $n \in \mathbb{N}$ we have that
\begin{equation*}
\E\braket{e^{-rT_n}y^{\varepsilon}_{T_n},p_{T_n}} - \E\int_0^{T_n} e^{-rt} \braket{y^{\varepsilon}_t, D_xf(\bar{X}_t,\bar{u}_t) }dt  = \sum_{j=1}^d\E\int_0^{T_n} e^{-rt} \braket{q^j_t, \delta \sigma^j(\bar{X}_t,\bar{u}_t)}dt.
\end{equation*}
Thanks to the growth assumptions on $\sigma, f$ and to the regularity of $y^{\varepsilon}_t$ and $q_t$, we can send $T_n$ to infinity to end with
\begin{equation}\label{eq:duality:yp}
\E\int_0^\infty e^{-rt} \braket{y^{\varepsilon}_t, D_xf(\bar{X}_t,\bar{u}_t) }dt  = -\sum_{j=1}^d\E\int_0^\infty e^{-rt} \braket{q^j_t, \delta \sigma^j(\bar{X}_t,\bar{u}_t)}dt.
\end{equation}

\noindent Repeating the same argument for $e^{-rt}\braket{z^{\varepsilon}_t,p_t}$ we get
\begin{equation*}
\begin{split}
d \Big( e^{-rt}\braket{z^{\varepsilon}_t,p_t} \Big) &= \braket{d(e^{-rt} z^{\varepsilon}_t),p_t} + e^{-rt}\braket{z_t,dp_t} \\
& \quad + \sum_{j=1}^d \braket{q^j_t, D_x\sigma^j(\bar{X}_t,\bar{u}_t)z^{\varepsilon}_t + \delta \sigma^j(\bar{X}_t,\bar{u}_t)y^{\varepsilon}_t + \frac{1}{2}D_x^2\sigma^j (y^{\varepsilon}_t)^2}dt \\
&= \Big[ -re^{-rt}\braket{z^{\varepsilon}_t,p_t} + e^{-rt}\braket{D_xb(\bar{X}_t,\bar{u}_t) z_t^{\varepsilon},p_t} + e^{-rt}\braket{\delta b(\bar{X}_t,\bar{u}_t), p_t}\\
& \quad + \frac{1}{2}e^{-rt}\braket{D_x^2b(\bar{X}_t,\bar{u}_t)(y^{\varepsilon}_t)^2,p_t} - e^{-rt}\braket{z^{\varepsilon}_t, D_xb(\bar{X}_t,\bar{u}_t)^T y_t^{\varepsilon}} \\
& \quad - \sum_{j=1}^d e^{-rt}\braket{z^{\varepsilon}_t, D_x\sigma^j(\bar{X}_t,\bar{u}_t)^T q^j_t} + e^{-rt}\braket{z^{\varepsilon}_t,D_xf(\bar{X}_t,\bar{u}_t)} \\
& \quad + re^{-rt}\braket{z^{\varepsilon}_t,p_t} + \sum_{j=1}^d \braket{q^j_t, D_x\sigma^j(\bar{X}_t,\bar{u}_t)z^{\varepsilon}_t}\\
& \quad + \sum_{j=1}^d \braket{q^j_t, \delta \sigma^j(\bar{X}_t,\bar{u}_t)y^{\varepsilon}_t + \frac{1}{2}D_x^2\sigma^j (y^{\varepsilon}_t)^2} \Big]dt + N_tdW_t,
\end{split}
\end{equation*}
where the stochastic term is a local martingale with zero mean value (which can be proved by same argument as before). Hence, taking expectation we obtain for all $T >0$
\begin{equation*}
\begin{split}
-\E\int_0^T e^{-rt} \braket{z^{\varepsilon}_t, D_xf(\bar{X}_t,\bar{u}_t) }dt  &= \E\int_0^T e^{-rt} \braket{\delta b(\bar{X}_t,\bar{u}_t) + \frac{1}{2}D_x^2b(\bar{X}_t,\bar{u}_t)(y^{\varepsilon}_t)^2 ,p_t} dt  \\
& \quad + \frac{1}{2}\sum_{j=1}^d \E\int_0^T e^{-rt} \braket{D_x^2\sigma^j(\bar{X}_t,\bar{u}_t)(y^{\varepsilon}_t)^2,q_t}dt + o(\varepsilon).
\end{split}
\end{equation*}
The term $o(\varepsilon)$ comes from the following estimate
\begin{equation*}
\begin{split}
\Big|\E&\int_0^{\infty} e^{-rt}\braket{\delta D_x\sigma^j(\bar{X}_t,\bar{u}_t) y^{\varepsilon}_t,q^j_t}dt\Big| \\
&\leq C\E\int_{E_{\varepsilon}}e^{-rt}(1 + \abs{\bar{X}_t}^m)\abs{y^{\varepsilon}_t}\abs{q^j_t} dt \\
&\leq C \int_{E_{\varepsilon}} \left(\E \left[ e^{-rt}(1 + \abs{\bar{X}_t}^m)^2\abs{y^{\varepsilon}_t}^2 \right] \right)^{1/2}\left(\E e^{-rt}\abs{q^j_t}^2 \right)^{1/2}dt \\
&\leq C \sup_{t \in E_{\varepsilon}}\left( \E e^{-rt}\abs{y^{\varepsilon}_t}^4 \right)^{1/4} \int_{E_{\varepsilon}}\left(\E e^{-rt}\abs{q^j_t}^2 \right)^{1/2}dt \\
&\leq C \varepsilon \left( \int_{E_{\varepsilon}}\E e^{-rt}\abs{q^j_t}^2 dt \right)^{1/2},
\end{split}
\end{equation*}
and the last integral goes to zero as $\varepsilon$ goes to zero, since $\E \int_0^{\infty} e^{-rt}\abs{q^j_t}^2 dt < \infty$.
Applying the same strategy as before we can choose a sequence $(T_n)_{n\geq 1}$
with $T_n \nearrow +\infty$ as $n \rightarrow +\infty$ such that along this sequence $\E \big[ e^{-rT_n}p_{T_{n}} \big] \to 0$.
This way we end up with
\begin{equation}\label{eq:duality:zp}
\begin{split}
-\E\int_0^\infty e^{-rt}& \braket{z^{\varepsilon}_t, D_xf(\bar{X}_t,\bar{u}_t) }dt \\
&= \E\int_0^\infty e^{-rt} \braket{\delta b(\bar{X}_t,\bar{u}_t) + \frac{1}{2}D_x^2b(\bar{X}_t,\bar{u}_t)(y^{\varepsilon}_t)^2 ,p_t} dt  \\
&\quad + \frac{1}{2}\sum_{j=1}^d \E\int_0^\infty e^{-rt} \braket{D_x^2\sigma^j(\bar{X}_t,\bar{u}_t)(y^{\varepsilon}_t)^2,q_t}dt + o(\varepsilon).
\end{split}
\end{equation}
If we substitute relations \eqref{eq:duality:yp} and \eqref{eq:duality:zp} into equation \eqref{eq:spike_cost1}, we obtain
\begin{equation*}
\begin{split}
J\lp u^{\varepsilon} (\cdot) \rp - J \lp \bar{u}(\cdot) \rp &= \E \int_0^{\infty} e^{-rt} \left[ -\sum_{j=1}^{d} \braket{q^j_t,\delta\sigma^j(\bar{X}_t,\bar{u}_t)} - \braket{p_t,\delta b(\bar{X}_t,\bar{u}_t)} + \delta f(\bar{X}_t,\bar{u}_t) \right]dt \\
& \quad + \frac{1}{2} \E \int_0^{\infty}e^{-rt}\left[ -\sum_{j=1}^d\braket{D_x^2\sigma^j(\bar{X}_t,\bar{u}_t)(y^{\varepsilon}_t)^2,q_t} - \braket{D_x^2 b(\bar{X}_t,\bar{u}_t)(y^{\varepsilon}_t)^2,p_t} \right.\\
& \qquad \qquad \Bigg. + \braket{D_x^2 f(\bar{X}_t,\bar{u}_t) y_t,y_t} \Bigg]dt + o(\varepsilon),
\end{split}
\end{equation*}
and recalling the definition of the Hamiltonian $H(x,u,p,q) = \braket{p,b(x,u)} + \opn{Tr}\left[ q^T \sigma(x,u) \right] - f(x,u)$ we have the desired result.
\end{proof}
\subsection{Conditions on the discount factor $r$}
Here we collect some restrictions on the discount factor used throughout the computations in the paper. For the purposes of the SMP it is not necessary to find precise values of the discount factor, in general $r$ has to be positive and big enough. Nevertheless, it can be useful to exhibit some sufficient conditions.  \medskip

Starting from the well posedness of the state equation, we have to require $r > 2c_{1/2}$ in order to find a unique solution in the space $L^{2,-r}_{\mathcal{F}}(\R_+;\R^n)$. Regarding the first variation equation we have no other restriction. On the contrary, to assure that $D^2_x b(\bar{X}_t,\bar{u}_t)(y^{\varepsilon}_t)^2,D^2_x \sigma^j(\bar{X}_t,\bar{u}_t)(y^{\varepsilon}_t)^2 \in L^{2,-r}_{\mathcal{F}}(\R_+;\R^n)$ in the equation for $z^{\varepsilon}$, a sufficient condition
is $r > 2\max \lbrace 0, c_{1/2}, c_{2(2m+1)-1}, c_{2m-1}, c_3 \rbrace$, where $c_3$ comes from estimate \eqref{eq:state estimate 1} applied to the process $y^{\varepsilon}$. Further restrictions come from the proof of Proposition \ref{p:expansion}. Here it follows that one can choose $\rho_1, \rho_2 \geq 2c_{1/2}$; $\rho_3,\rho_5 \geq 16k(2m+1)\max \{c_{2k(2m+1)-1}, c_{8k-1}, c_{2km-1}\}$
 and $\rho_4 \geq 64k(2m+1)\max \{c_{4k(2m+1)-1}, c_{4km-1}\}$. These conditions are derived from the polynomial growth assumptions and from the use of the H\"older inequality.

The choice of the discount factor $r$ for the first adjoint equation (see Theorem \ref{t:adjoint:eq}) depends on the a priori estimate given by Lemma \ref{l:apriori.backward} as well as the integrability of the forcing term $D_xf(\bar{X}_t,\bar{u}_t)$. Therefore, due to the polynomial growth, it is easy to see that it is sufficient to consider $r > 2\max \lbrace 0, c_{1/2}, c_{l-1}\rbrace $.

For the existence and uniqueness of $(y^{t,\eta}_s)$, we choose $r > 2c_{1/2}$. Regarding the estimates \eqref{eq:estimate_y_gamma} and \eqref{eq:continuity:y:initial_data}, it is sufficient to choose $r > 2 \max \{c_{1/2},c_3\}$. Now, for the existence of the process $P$, it is sufficient to take $r> 2\max\{0,c_5,c_{3(2m+1)-1},c_{3m-1},c_{3l-1}\}$ (for $p=\frac{3}{2}, q =3$ in \eqref{eq:P bound estimate}). Regarding Proposition \ref{p:last:estimtes}, we have to add some restrictions originating from Lemma \ref{l:4.1} throughout the proof. More precisely, it is sufficient to require $r > 2 \max \lbrace c_{7}, c_{2m-1}, c_3 \rbrace$. \smallskip

To conclude, the statement of the SMP holds true if the discount factor is chosen in a way such that all the previous results can be applied.  Hence, it is sufficient to choose $k=1$ and $r$ such that
$r > 64(2m+1)\max \lbrace 0,c_{1/2}, c_3, c_5, c_7, c_{l-1}, c_{3l-1}, c_{2m-1},c_{3m-1}, c_{4m-1}, c_{2(2m+1)-1}, c_{3(2m+1)-1}, c_{4(2m+1)-1} \rbrace.$

\bibliography{mybib}

\begin{thebibliography}{10}

\bibitem{bahlali2005general}
Seid Bahlali and Brahim Mezerdi.
\newblock A general stochastic maximum principle for singular control problems.
\newblock {\em Electronic Journal of Probability}, 10(30):988--1004, 2005.

\bibitem{Briand2003109}
Ph. Briand, B.~Delyon, Y.~Hu, E.~Pardoux, and L.~Stoica.
\newblock Lp solutions of backward stochastic differential equations.
\newblock {\em Stochastic Processes and their Applications}, 108(1):109 -- 129,
  2003.

\bibitem{cerrai2001second}
Sandra Cerrai.
\newblock {\em Second order {PDE}'s in finite and infinite dimension}, volume
  1762 of {\em Lecture Notes in Mathematics}.
\newblock Springer-Verlag, Berlin, 2001.
\newblock A probabilistic approach.

\bibitem{cox2013local}
Sonja Cox, Martin Hutzenthaler, and Arnulf Jentzen.
\newblock Local lipschitz continuity in the initial value and strong
  completeness for nonlinear stochastic differential equations.
\newblock {\em preprint arXiv:1309.5595}, 2013.

\bibitem{du2013maximum}
Kai Du and Qingxin Meng.
\newblock A maximum principle for optimal control of stochastic evolution
  equations.
\newblock {\em SIAM J. Control Optim.}, 51(6):4343--4362, 2013.

\bibitem{dufour2006maximum}
Dufour F. and Miller B.
\newblock Maximum principle for singular stochastic control problems.
\newblock {\em SIAM J. Control Optim.}, 45(2):668--698, 2006.

\bibitem{fuhrman2013stochastic}
Marco Fuhrman, Ying Hu, and Gianmario Tessitore.
\newblock Stochastic maximum principle for optimal control of {SPDE}s.
\newblock {\em Appl. Math. Optim.}, 68(2):181--217, 2013.

\bibitem{lan2014new}
Guangqiang Lan and Jiang-Lun Wu.
\newblock New sufficient conditions of existence, moment estimations and non
  confluence for {SDE}s with non-{L}ipschitzian coefficients.
\newblock {\em Stochastic Process. Appl.}, 124(12):4030--4049, 2014.

\bibitem{lu2012general}
Qi~L{\"u} and Xu~Zhang.
\newblock General pontryagin-type stochastic maximum principle and backward
  stochastic evolution equations in infinite dimensions.
\newblock {\em preprint arXiv:1204.3275}, 2012.

\bibitem{maslowski2014sufficient}
Bohdan Maslowski and Petr Veverka.
\newblock Sufficient stochastic maximum principle for discounted control
  problem.
\newblock {\em Appl. Math. Optim.}, 70(2):225--252, 2014.

\bibitem{oksendal2005applied}
Bernt {\O}ksendal and Agn{\`e}s Sulem.
\newblock {\em Applied stochastic control of jump diffusions}, volume 498.
\newblock Springer, 2005.

\bibitem{oksendal2011singular}
Bernt {\O}ksendal and Agn\`es Sulem.
\newblock {Singular stochastic control and optimal stopping with partial
  information of It{\^o}--L{\'e}vy processes}.
\newblock Research Report RR-7708, August 2011.

\bibitem{oksendal2011optimal}
Bernt {\O}ksendal, Agn{\`e}s Sulem, and Tusheng Zhang.
\newblock Optimal control of stochastic delay equations and time-advanced
  backward stochastic differential equations.
\newblock {\em Adv. in Appl. Probab.}, 43(2):572--596, 03 2011.

\bibitem{ECP2548}
Martin Ondrej\'at and Jan Seidler.
\newblock On existence of progressively measurable modifications.
\newblock {\em Electron. Commun. Probab.}, 18:no. 20, 1--6, 2013.

\bibitem{orrieri2013stochastic}
Carlo Orrieri.
\newblock A stochastic maximum principle with dissipativity conditions.
\newblock {\em Disc. Cont. Dyn. Sist. A, to appear}.

\bibitem{pardoux1999bsdes}
{\'E}tienne Pardoux.
\newblock B{SDE}s, weak convergence and homogenization of semilinear {PDE}s.
\newblock In {\em Nonlinear analysis, differential equations and control
  ({M}ontreal, {QC}, 1998)}, volume 528 of {\em NATO Sci. Ser. C Math. Phys.
  Sci.}, pages 503--549. Kluwer Acad. Publ., Dordrecht, 1999.

\bibitem{peng1990general}
Shige Peng.
\newblock A general stochastic maximum principle for optimal control problems.
\newblock {\em SIAM J. Control Optim.}, 28(4):966--979, 1990.

\bibitem{peng2000infinite}
Shige Peng and Yufeng Shi.
\newblock Infinite horizon forward-backward stochastic differential equations.
\newblock {\em Stochastic Process. Appl.}, 85(1):75--92, 2000.

\bibitem{prevot2007concise}
Claudia Pr{\'e}v{\^o}t and Michael R{\"o}ckner.
\newblock {\em A concise course on stochastic partial differential equations},
  volume 1905 of {\em Lecture Notes in Mathematics}.
\newblock Springer, Berlin, 2007.

\bibitem{tang1993maximum}
Shan~Jian Tang and Xun~Jing Li.
\newblock Maximum principle for optimal control of distributed parameter
  stochastic systems with random jumps.
\newblock In {\em Differential equations, dynamical systems, and control
  science}, volume 152 of {\em Lecture Notes in Pure and Appl. Math.}, pages
  867--890. Dekker, New York, 1994.

\bibitem{tang1994necessary}
Shan~Jian Tang and Xun~Jing Li.
\newblock Necessary conditions for optimal control of stochastic systems with
  random jumps.
\newblock {\em SIAM Journal on Control and Optimization}, 32(5):1447--1475,
  1994.

\bibitem{tessitore1996existence}
Gianmario Tessitore.
\newblock Existence, uniqueness and space regularity of the adapted solutions
  of a backward {SPDE}.
\newblock {\em Stochastic Anal. Appl.}, 14(4):461--486, 1996.

\bibitem{wu2012maximum}
Zhen Wu and Feng Zhang.
\newblock Maximum principle for stochastic recursive optimal control problems
  involving impulse controls.
\newblock {\em Abstr. Appl. Anal.}, 2012:1--16, 2012.

\bibitem{Yo99}
Jiongmin Yong and Xun~Yu Zhou.
\newblock {\em Stochastic Controls: Hamiltonian Systems and HJB Equations}.
\newblock New York: Springer-Verlag, 1999.

\bibitem{zhou1998stochastic}
Xun~Yu Zhou.
\newblock {Stochastic near-optimal controls: Necessary and sufficient
  conditions for near-optimality.}
\newblock {\em {SIAM J. Control Optim.}}, 36(3):929--947, 1998.

\end{thebibliography}
\bibliographystyle{plain}

\end{document}